\documentclass[11pt]{article}

\usepackage{amsmath,amsfonts,amssymb}
\usepackage{bbm}

\usepackage{lipsum}
\usepackage{kotex}
\usepackage{cancel}

\usepackage{xcolor}
\usepackage{ulem}

\newtheorem{defn}{Definition}[section]
\newtheorem{theorem}[defn]{Theorem}
\newtheorem{lemma}[defn]{Lemma}
\newtheorem{proposition}[defn]{Proposition}
\newtheorem{cor}[defn]{Corollary}
\newtheorem{remark}[defn]{Remark}
\newtheorem{example}[defn]{Example}

\newenvironment{proof}{{\bf Proof }}{{\vskip 0.1cm \hfill$\Box$}}

\def\N {{\mathbb N}}

\def\R {{\mathbb R}}

\begin{document}

\noindent
{{\Large\bf Weighted Helmholtz--Hodge decompositions, Lyapunov functions, and invariant measures}}\\[10pt]
\bigskip
\noindent
{\bf Haesung Lee}, {\bf Gerald Trutnau}\\
\noindent
{\bf Abstract.} 
We study weighted Helmholtz--Hodge decompositions of drift vector fields associated with second-order diffusion operators on $\R^d$, $d\ge 2$. Given a decomposition of the form
\[
  \mathbf{G}=A\nabla\Phi+\mathbf{B},
\]
we relate the weighted divergence-free condition $\mathrm{div}_{\mu}(\mathbf{B})=0$, where $\mu=e^{2\Phi}dx$, to infinitesimal invariance of $\mu$ for the operator
\[
  \frac12 \mathrm{trace}(A\nabla^2)+\langle \mathbf{G},\nabla\cdot\rangle.
\]
We compare weighted, orthogonal, and strictly orthogonal Helmholtz--Hodge decompositions and show that uniqueness of the infinitesimally invariant measure yields uniqueness of the corresponding weighted decomposition, and hence a canonical potential. For linear vector fields, we characterize Gaussian infinitesimally invariant measures by an algebraic Riccati equation together with a trace condition. In the Ornstein--Uhlenbeck case, this gives a structural proof of the classical criterion that a finite invariant measure exists if and only if the drift matrix is Hurwitz, and it identifies the associated strictly orthogonal decomposition. Finally, we treat nonlinear polynomial perturbations that preserve a given potential and obtain explicit classes of drifts for which the invariant measure and the weighted decomposition remain unique. The results clarify the relation between Lyapunov-type potentials, non-reversible perturbations, and invariant measures for diffusion semigroups.

\section{Introduction} 
The construction of Lyapunov functions and the identification of invariant measures are two
central problems in the analysis of dissipative dynamical systems and diffusion semigroups.
Although both topics have been studied extensively, they are often treated by rather different
methods. On the one hand, Lyapunov functions are usually introduced in order to detect stability,
asymptotic behavior, or domains of attraction. On the other hand, invariant measures are typically
studied from the viewpoint of Markov semigroups, partial differential equations, or stochastic
processes. The purpose of this paper is to show that, for a broad class of diffusion-type dynamics, Lyapunov-type potentials and invariant measures can be studied within a common framework based on weighted Helmholtz--Hodge decompositions.\\[3pt]
Throughout let $d\ge 2$. We first consider the second-order partial differential operator
\begin{eqnarray}\label{thegeneratora}
\frac12 \Delta  +\langle \mathbf{G},\nabla  \cdot\rangle
\end{eqnarray}
where $\mathbf{G} \in L^p_{loc}(\mathbb{R}^d, \mathbb{R}^d)$, for some $p \in (d, \infty)$ and $\langle \cdot,  \cdot\rangle$ denotes the Euclidean inner product. We are interested in decompositions of the vector field $ \mathbf{G}$ into an explicit gradient part and a complementary part. It follows from \cite[Theorem 1(i)]{BRS} (see also \cite{LT19} for some extensions) that $\mathbf{G}$ admits a weighted Helmholtz--Hodge decomposition (WHHD, Definition \ref{helmdefn}(iii)), namely there exists $\Phi \in H^{1,p}_{loc}(\mathbb{R}^d)$, such that with $\mu:= \exp(2\Phi)dx$
\begin{eqnarray}\label{gradientVFdecomp}
\mathbf{G}=\nabla \Phi+ (\mathbf{G}-\nabla \Phi), \quad \text{and} \quad \mathrm{div}_{\mu}(\mathbf{G}-\nabla \Phi)=0.
\end{eqnarray}
In general, however, $\rho =\exp(2\Phi) \in H^{1,p}_{loc}(\mathbb{R}^d)$ only exists as an abstract limit, i.e. as a limit along an unknown subsequence, of solutions to local variational equations (see \cite[Theorem 2.27]{LST22}). Therefore, the decomposition in \eqref{gradientVFdecomp} is, at this level of generality, not explicit. The main goal of this work is to investigate situations in which such decompositions are explicit, in the sense that the potential function $\Phi$ is known as a concrete function, and to exploit the resulting structural consequences. This has consequences for invariant measures and ergodicity, Lyapunov stability theory for ODEs and dynamical systems, and the geometric interpretation of drift decompositions into radial and tangential components.\\[3pt]
The weighted decomposition \eqref{gradientVFdecomp} is the most general decomposition considered in this paper. It leads to infinitesimal invariance of $\frac12 \Delta  +\langle \mathbf{G},\nabla  \cdot\rangle$ with respect to $\mu$ (Definition \ref{definvaran} and Theorem \ref{propstrohelm}(i)), and makes $\mu$ an explicit candidate for an invariant measure. If the divergence-free condition in \eqref{gradientVFdecomp} is taken with respect to Lebesgue measure, \eqref{gradientVFdecomp} reduces to the usual Helmholtz--Hodge decomposition, a fundamental tool in fluid dynamics, particularly in the study of incompressible Stokes and Navier--Stokes equations; see, for instance, \cite{S01,G11}. By contrast, the stochastic dynamics generated by \eqref{thegeneratora} naturally lead to infinitesimally invariant measures that are typically not Lebesgue measure. For example,  in the ergodic case the corresponding invariant measures are finite, whereas Lebesgue measure is infinite on $\mathbb R^d$. Thus \eqref{gradientVFdecomp} should in general be understood as a weighted Helmholtz--Hodge decomposition, rather than as the usual Helmholtz--Hodge decomposition. In the present case $A=id$, the case $\mathbf{G}=\nabla\Phi$ corresponds to reversible dynamics, while the addition of a non-zero weighted divergence-free part typically produces a non-reversible perturbation with the same infinitesimally invariant measure. Therefore, a broad family of gradient vector field decompositions as given by \eqref{gradientVFdecomp} describes a whole landscape of non-reversible perturbations of a given reversible system.\\[3pt]
In a next step, we consider the more specific orthogonal Helmholtz--Hodge decomposition (OHHD, Definition \ref{helmdefn}(ii))
\begin{eqnarray}\label{gradientVFdecompo}
\mathbf{G}=\nabla \Phi+ (\mathbf{G}-\nabla \Phi), \quad \mathrm{div}_{dx }(\mathbf{G}-\nabla \Phi)=0\quad \text{and} \quad \langle \nabla \Phi, (\mathbf{G}-\nabla \Phi) \rangle = 0,
\end{eqnarray}
and its pointwise version, the strictly orthogonal Helmholtz--Hodge decomposition (SOHHD), under stronger regularity assumptions (cf. end of Definition \ref{helmdefn}). Such decompositions are classical Helmholtz--Hodge decompositions with an additional orthogonality condition and automatically imply the weighted decomposition \eqref{gradientVFdecomp} above (see Theorem \ref{propstrohelm}). Geometrically, the decomposition \eqref{gradientVFdecompo} splits the drift into a part normal to the level sets of the potential and a divergence-free part tangent to these level sets. In this situation the potential is not merely an auxiliary scalar function, but a natural object associated with the dynamics. This explains the relevance of OHHDs and SOHHDs in the study of Lyapunov functions and domains of attraction. SOHHDs have recently been studied in \cite{Sud1, Sud2, Liu22}, mainly for linear and planar systems, with some outlook toward nonlinear systems, and they also occur in related problems from statistical mechanics and metastability. At the same time, orthogonal decompositions are restrictive: simple examples show that stability and even uniqueness of the weighted decomposition need not imply existence of an OHHD; see Example \ref{counterexamplestrong} below.\\[3pt]
The same geometric picture has a stochastic counterpart when the deterministic system is replaced by a diffusion with generator \eqref{thegeneratora}, or more generally by a diffusion operator with second-order part $\frac12\mathrm{trace}(A\nabla^2)$, where $A$ is symmetric and positive definite. In particular, this connects our work to perturbations that arise naturally in the study of accelerated convergence to equilibrium, metastability, and non-reversible sampling algorithms (see, for instance, \cite{MCF15} and the more specific references cited below).  In that context one considers perturbations of a reversible diffusion by additional drifts that are weighted divergence-free with respect to the equilibrium density, so that the invariant measure is kept fixed while the long-time behavior may improve. We mention the papers \cite{LPM19,LI22}, where SOHHDs occur in the study of metastability of non-reversible diffusions and \cite{HHS05} where WHHDs are studied in the context of accelerating diffusions. The present work fits naturally into this circle of ideas, but with a different emphasis. Rather than starting from convergence-rate comparison alone, we focus on the decomposition itself and on the structural relation between the potential, the drift perturbation, and the invariant measure. In particular, our framework provides explicit classes of perturbations for which the invariant measure can be written down directly, and it clarifies when such perturbations come from orthogonal, strictly orthogonal, or merely weighted Helmholtz--Hodge structures.\\[3pt]
A further theme is uniqueness (see Theorem \ref{nonexofOHHD} and Remark \ref{rem:SOHHD_exuniq} for the main general results in this direction). From the viewpoint of Lyapunov stability alone, uniqueness of a Lyapunov function is not essential. By contrast, if the infinitesimally invariant measure is unique, then the potential in a weighted Helmholtz--Hodge decomposition becomes canonical, up to the natural additive normalization. We prove that uniqueness of the infinitesimally invariant measure implies uniqueness of the weighted decomposition, and we identify when the unique weighted decomposition is also orthogonal. \\
In Sections \ref{sec: WHHD and linear} and \ref{sec: WHHD and nonlinear}, we study transformations of the drift which preserve the invariant-measure structure and the Lyapunov potential; see, for instance, Theorem \ref{thm:polydifdrift}, Proposition \ref{prop:lyap-criterion}, Lemma \ref{conservlema}, and Theorem \ref{mainthm}.\\[3pt]
A second main part of the paper concerns linear vector fields and Ornstein--Uhlenbeck type operators. For
\[
L^{A,G}f=\frac12 \operatorname{trace}(A\nabla^2 f)+\langle Gx,\nabla f\rangle,
\]
as defined in \eqref{LAG} below, we characterize Gaussian infinitesimally invariant measures $\mu=e^{\langle x,Sx\rangle}dx$ (cf. Proposition \ref{propintegrep}(i)) by the algebraic Riccati equation
\[
SG+G^T S=2SAS
\]
together with the trace condition $\operatorname{trace}(G-AS)=0$. In the Hurwitz case this yields the unique finite invariant measure of the Ornstein--Uhlenbeck semigroup and the corresponding strictly orthogonal decomposition. Although the invariant measure criterion for Ornstein--Uhlenbeck semigroups is classical (\cite{LB07,LMP20,DZ92}), the proof given here (based on tools developed in \cite{Sud1,Sud2}, and a method that was used in \cite[Section 4.]{Liu22} for $A=id$ to solve the algebraic Riccati equation \eqref{SOHHDlinalg1}) exposes a direct structural link between Lyapunov-type potentials, Gaussian invariant measures, Helmholtz--Hodge decompositions, and algebraic Riccati equations. \\[3pt]
After the linear theory, we study nonlinear, in particular polynomial, perturbations. The guiding question is when a potential obtained from the linear problem remains valid after adding nonlinear drift terms. We obtain verifiable algebraic and growth conditions under which the weighted decomposition, the Lyapunov function, and the invariant measure remain explicit; in favorable cases the decomposition is strictly orthogonal, while in others it is only weighted but still unique.\\[3pt]
The main contributions can be summarized as follows.
\begin{enumerate}
\item We formulate weighted, orthogonal, and strictly orthogonal Helmholtz--Hodge decompositions for diffusion drifts and relate the weighted divergence-free condition to infinitesimal invariance.
\item We prove uniqueness of the weighted decomposition under uniqueness of the infinitesimally invariant measure, and we derive consequences for orthogonal decompositions.
\item For linear drifts we characterize Gaussian infinitesimally invariant measures through an algebraic Riccati equation with a trace condition, recovering the Hurwitz criterion for finite Ornstein--Uhlenbeck invariant measures from this structural viewpoint.
\item We construct nonlinear polynomial perturbations preserving explicit potentials and invariant measures, and distinguish when the resulting decompositions are strictly orthogonal or only weighted.
\end{enumerate}
The paper is organized as follows. In Section \ref{sec: WHHD and linear} we introduce weighted, orthogonal, and strictly orthogonal Helmholtz--Hodge decompositions and establish their relation to infinitesimally invariant measures. We then analyze linear vector fields and derive the Riccati characterization of Gaussian invariant measures, including the Ornstein--Uhlenbeck case. In Section \ref{sec: WHHD and nonlinear} we study nonlinear polynomial perturbations and obtain criteria for explicit invariant measures and uniqueness of the associated decompositions. Section \ref{section:algRiceq} develops the framework of algebraic Riccati equations needed in this work. The appendix is devoted to concrete examples and to additional structural properties illustrating both the scope and the limitations of the proposed framework.

\section{Weighted Helmholtz--Hodge decompositions and linear vector fields}\label{sec: WHHD and linear}
\subsection{Notation and operator}\label{sec: Not and Op}
Let $A=(a_{ij})_{1\le i,j\le d}$ be a symmetric positive definite matrix of real numbers and $\mathbf{G}\in L^2_{loc}(\R^d, \R^d)$ be a given vector field. Throughout this paper, we will consider
\begin{eqnarray}\label{thegenerator}
L^{A,\mathbf{G}}f:=\frac12 \mathrm{trace}(A\nabla^2 f) +\langle \mathbf{G},\nabla  f\rangle,
\end{eqnarray}
where $f$ has suitable regularity. If $A=id$, we simply write $L^{\mathbf{G}}$ instead of $L^{id,\mathbf{G}}$, hence $L^{\mathbf{G}}f:=\frac12 \Delta f +\langle \mathbf{G},\nabla  f\rangle$. Here we consider $\mathbf{G}$ as a vector field. Thus for a $d\times d$ matrix $G$ of real numbers, the notation $L^{A,G}$ will make perfectly sense as $x\mapsto Gx$ describes a vector field. However, in the sequel, we often write $Gx$ instead of $G$, for instance by writing $L^{A,Gx}$ or the like,  to emphasize the special character of matrix vector fields.  \\
For a matrix $B=(b_{ij})_{1\le i,j\le d}$ of weakly partially differentiable functions, we define (cf. \cite{LT25})
$$
\mathrm{div} B:=\nabla B^T:=\Big (\sum_{j=1}^{d}\partial_j b^T_{ij}\Big )_{1\le i\le  d}
\ , 
$$
where $B^T=(b^T_{ij})_{1 \leq i,j \leq d}=(b_{ji})_{1 \leq i,j \leq d}$ denotes the transpose of $B$. For $\rho  \in H^{1,2}_{loc}(\mathbb{R}^d) \cap L^{\infty}_{loc}(\mathbb{R}^d)$ with $\rho>0$ a.e., we write
\begin{equation*}\label{logder 2}
\beta^{\rho, B^T} := \frac12\mathrm{div} B + \frac{1}{2\rho} B^T \nabla \rho.
\end{equation*}
\begin{defn} \label{definvaran}
A positive, locally finite measure $\widehat{\mu}$ defined on $(\R^d, \mathcal{B}(\R^d))$ satisfying $L^{A,\mathbf{G}}f \in L^1(\R^d, \widehat{\mu})$ for all $f \in C_0^{\infty}(\R^d)$, where $L^{A,\mathbf{G}}$ is as in \eqref{thegenerator},
is said to be an infinitesimally invariant measure for $(L^{A,\mathbf{G}},C_0^{\infty}(\R^d))$, if
$$
\int_{\R^d} L^{A,\mathbf{G}}f d\widehat{\mu} = 0, \qquad \forall f \in C_0^{\infty}(\R^d).
$$
\end{defn}
\begin{remark}\label{definvaranrem}
If $\widehat{\mu}$ is an infinitesimally invariant measure for $(L^{A,\mathbf{G}}, C_0^{\infty}(\R^d))$, where $L^{A,\mathbf{G}}$ is as in \eqref{thegenerator} with $\mathbf{G}\in L^p_{loc}(\R^d, \R^d)$ for some $p\in(d,\infty)$, then by \cite[Corollary 2.10, 2.11]{BKR2}, there exists $\widehat{\rho} \in H^{1,p}_{loc}(\R^d) \cap C(\R^d)$ such that $\widehat{\rho}(x)>0$ for all $x \in \R^d$ and $\widehat{\mu}= \widehat{\rho} dx$. 
\end{remark}

\subsection{Weighted Helmholtz--Hodge decompositions}
\begin{defn} \label{helmdefn}
Let $\widetilde{\mathbf{G}} \in L^2_{loc}(\mathbb{R}^d, \mathbb{R}^d)$.

\begin{itemize}
\item[(i)]
We say that $\widetilde{\mathbf{G}} $ admits a Helmholtz--Hodge decomposition (in short: {\bf HHD}), if there exist $\Phi \in H^{1,2}_{loc}(\mathbb{R}^d) \cap L^{\infty}_{loc}(\mathbb{R}^d)$ and $\mathbf{B} \in L^2_{loc}(\mathbb{R}^d, \mathbb{R}^d)$, such that 
\begin{equation} \label{divfreecon}
\int_{\mathbb{R}^d} \langle \mathbf{B}, \nabla f \rangle \,dx = 0 \quad \forall f \in C_0^{\infty}(\mathbb{R}^d) \quad \text{(in short: $\mathrm{div}_{dx }(\mathbf{B})=0$)}
\end{equation}
 and
\begin{equation} \label{mathbfbolb}
\widetilde{\mathbf{G}} = \nabla \Phi + \mathbf{B}\quad (\text{a.e. on } \  \mathbb{R}^d).
\end{equation}
\item[(ii)]
We say that $\widetilde{\mathbf{G}} $ admits an orthogonal Helmholtz--Hodge decomposition (in short: {\bf OHHD}), 
if $\widetilde{\mathbf{G}} $ admits a HHD as in (i), and in addition,
\begin{equation} \label{BorthnablaPhi}
\langle \nabla \Phi, \mathbf{B} \rangle = 0 \quad (\text{a.e. on } \  \mathbb{R}^d).
\end{equation}
\item[(iii)]
We say that $\widetilde{\mathbf{G}} $ admits a weighted Helmholtz--Hodge decomposition (in short: {\bf WHHD}), if   there exist $\Phi \in H^{1,2}_{loc}(\mathbb{R}^d) \cap L^{\infty}_{loc}(\mathbb{R}^d)$ and $\mathbf{B} \in L^2_{loc}(\mathbb{R}^d, \mathbb{R}^d)$ such that
\eqref{mathbfbolb} holds and
$$
\int_{\mathbb{R}^d} \langle \mathbf{B}, \nabla f \rangle \,d\mu = 0 \quad \forall f \in C_0^{\infty}(\mathbb{R}^d) \quad \text{(in short: $\mathrm{div}_{\mu }(\mathbf{B})=0$)},
$$
where $\mu=\exp(2\Phi)\,dx$.
\end{itemize}
If $\widetilde{\mathbf{G}} $ is a continuous vector field admitting an OHHD, and  $\Phi$ as well as  $\mathbf{B}$ are continuously differentiable, then we say that $\widetilde{\mathbf{G}} $ admits a strictly orthogonal Helmholtz--Hodge decomposition (in short: {\bf SOHHD}). Note that in this case, \eqref{mathbfbolb}, \eqref{BorthnablaPhi}, and $\mathrm{div}\mathbf{B}=0$ hold everywhere on $\R^d$.
\end{defn}
\begin{remark}\label{rem: 2.4}
The regularity of $\Phi$ in Definition \ref{helmdefn} will be sufficient for the purposes of this paper. It can be relaxed to 
$$
\Phi\in H^{1,1}_{loc}(\mathbb{R}^d), \qquad \text{such that }\quad \exp(2\Phi)(1+\|\nabla \Phi\|^2)\in L^1_{loc}(\mathbb{R}^d).
$$ 
Then $\varphi:=\exp(\Phi)\in H^{1,2}_{loc}(\mathbb{R}^d)$ and $\mu =\varphi^2 dx$ has full support, so the framework of \cite[Section 2.1.1]{LST22} can be applied.
The only slight modification would be the assumption $\mathbf{B} \in L^2_{loc}(\mathbb{R}^d, \mathbb{R}^d,\mu)$, since then possibly $L^2_{loc}(\mathbb{R}^d, \mathbb{R}^d)\not =L^2_{loc}(\mathbb{R}^d, \mathbb{R}^d,\mu)$.
\end{remark}
\begin{theorem} \label{propstrohelm}
Consider
$$
L^{A,\mathbf{G}}f:=\frac12 \mathrm{trace}(A\nabla^2 f) +\langle \mathbf{G},\nabla  f\rangle, \quad f \in C_0^{\infty}(\mathbb{R}^d),
$$
as in \eqref{thegenerator}, and suppose that
\begin{eqnarray}\label{graddec}
\mathbf{G} =A\nabla \Phi + \mathbf{B}, \quad \text{with}\ \Phi \in H^{1,2}_{loc}(\mathbb{R}^d) \cap L^{\infty}_{loc}(\mathbb{R}^d), \ \mathbf{B} \in L^2_{loc}(\mathbb{R}^d, \mathbb{R}^d).
\end{eqnarray}
Let $\mu=\rho\,dx$, where $\rho=\exp(2 \Phi)$. Then, it holds:
\begin{itemize}
\item[(i)]
$\mu$ is an infinitesimally invariant measure for $(L^{A,\mathbf{G}}, C_0^{\infty}(\mathbb{R}^d))$, if and only if $
{\mathrm{div}}_{\mu}(\mathbf{B})=0$.
\item[(ii)] 
\begin{equation} \label{equidivfre}
\mathrm{div}_{\mu}(\mathbf{B})=0\quad \Longleftrightarrow \quad \int_{\mathbb{R}^d}\big (\langle \nabla \Phi, \mathbf{B} \rangle  \varphi -\frac 12\big \langle  \mathbf{B}, \nabla \varphi \big \rangle\big ) \,dx = 0 \quad \forall \varphi \in C_0^{\infty}(\mathbb{R}^d).
\end{equation}
In particular, if $\mathbf{B}$ is weakly differentiable, then using integration by parts in \eqref{equidivfre}, $\mathrm{div}_{\mu}(\mathbf{B})=0$ is equivalent to 
$$
\langle \nabla \Phi, \mathbf{B} \rangle+\frac12\mathrm{div}\mathbf{B}=0.
$$
In general, in the case where $\mathbf{B}$ is not necessarily weakly differentiable, \eqref{equidivfre} implies 
\begin{equation} \label{equidivfre2}
\mathrm{div}_{\mu}(\mathbf{B})=0\quad\text{and} \quad  \mathrm{div}_{dx}(\mathbf{B})=0\quad \Longrightarrow \quad \langle \nabla \Phi, \mathbf{B} \rangle = 0 
\end{equation}
and 
\begin{equation} \label{equidivfre3}
\mathrm{div}_{dx }(\mathbf{B})=0\quad \text{and} \quad \langle \nabla \Phi, \mathbf{B} \rangle = 0 
\quad \Longrightarrow \quad \mathrm{div}_{\mu}(\mathbf{B})=0,
\end{equation}
and
\begin{equation} \label{equidivfreadditi}
\mathrm{div}_{\mu}(\mathbf{B})=0\quad\text{and} \quad  \langle \nabla \Phi, \mathbf{B} \rangle = 0 
\quad \Longrightarrow \quad \mathrm{div}_{dx}(\mathbf{B})=0. 
\end{equation}\\
Moreover, it follows from \eqref{equidivfre3} that if the decomposition of  
\begin{eqnarray}\label{propstrohelm:eqOHHD}
\widetilde{\mathbf{G}} = \nabla \Phi + \mathbf{B},
\end{eqnarray}
(where $\mathbf{B}=\mathbf{G} -A\nabla \Phi $ according to \eqref{graddec}) is an OHHD as in Definition \ref{helmdefn}, then it is a WHHD. Conversely, it follows from \eqref{equidivfre} that if the decomposition in \eqref{propstrohelm:eqOHHD} is a WHHD, then it is only an OHHD if either $\mathrm{div}_{dx}(\mathbf{B})=0$ or $\langle \nabla \Phi, \mathbf{B} \rangle = 0$, in which case both are zero, cf. \eqref{equidivfre2} and \eqref{equidivfreadditi}.
\end{itemize}
\end{theorem}
\begin{proof} 
(i)
Note that
$$
L^{A,\mathbf{G}}f=\frac12 \mathrm{trace}(A\nabla^2 f) + \left\langle A\nabla \Phi, \nabla f \right\rangle + \langle \mathbf{B}, \nabla f \rangle, \quad f \in C_0^{\infty}(\mathbb{R}^d).
$$
Since $L^{A,A\nabla \Phi}=\frac12 \mathrm{trace}(A\nabla^2 f)  + \left\langle A \nabla \Phi, \nabla \cdot\right\rangle$ is symmetric with respect to $\mu$, it follows that $\mu$ is infinitesimally invariant for $(L^{A,A\nabla \Phi}, C_0^{\infty}(\mathbb{R}^d))$, and therefore $\mathrm{div}_{\mu}(\mathbf{B})=0$ is equivalent to $\mu$ being infinitesimally invariant for $(L^{A,\mathbf{G}}, C_0^{\infty}(\mathbb{R}^d))$.\\
(ii) Observe that \eqref{equidivfre} holds, since for all $\tilde{\varphi} \in H^{1,2}(\mathbb{R}^d)_0 \cap L^{\infty}(\mathbb{R}^d)$,
\begin{align*}
\int_{\mathbb{R}^d} \langle \mathbf{B}, \nabla \tilde{\varphi} \rangle \, d\mu 
&= \int_{\mathbb{R}^d} \langle \mathbf{B}, \nabla (\rho \tilde{\varphi}) \rangle \, dx  - \int_{\mathbb{R}^d} \langle \mathbf{B}, \nabla \rho \rangle \tilde{\varphi} \, dx  \\
&=	\int_{\mathbb{R}^d} \langle \mathbf{B}, \nabla (\rho \tilde{\varphi}) \rangle \, dx  - \int_{\mathbb{R}^d} \langle \mathbf{B}, 2\nabla \Phi \rangle \rho \tilde{\varphi} \, dx.
\end{align*}
Since every function in $H^{1,2}(\mathbb R^d)_0\cap L^\infty(\mathbb R^d)$ can be approximated in $H^{1,2}(\mathbb R^d)$ by a uniformly bounded sequence in $C_0^\infty(\mathbb R^d)$, the assertion follows by density.
\end{proof}
\begin{remark}\label{connectionODE OHHD}(Interpretation of OHHDs and SOHHDs.)
Let $\mathbf{G}$ be smooth and consider the deterministic system
\[
  \frac{dX_t}{dt}=\mathbf{G}(X_t).
\]
If $-\Phi$ is a Lyapunov function for this system, then
\[
  \langle \nabla(-\Phi(x)),\mathbf{G}(x)\rangle \leq 0
\]
on the region where the Lyapunov condition is imposed. This condition only says that the drift does not increase $-\Phi$ along trajectories; it does not determine how the non-gradient part of the drift is oriented relative to the level sets of $\Phi$. If, however,
\[
  G=\nabla\Phi+\mathbf{B}, \qquad \langle \nabla\Phi,\mathbf{B}\rangle=0,
\]
then $\mathbf{B}$ is tangent to the level sets of $\Phi$, and hence also to the level sets of $-\Phi$. Consequently
\[
\langle \nabla (-\Phi(x)), \mathbf{G}(x) \rangle=-\langle \nabla \Phi(x), \nabla \Phi(x) \rangle\le 0.
\]
Thus an OHHD gives a particularly transparent Lyapunov structure: the gradient part changes the potential, while the orthogonal part moves tangentially along its level sets. If, in addition, the decomposition is strict, then these identities hold pointwise and the condition $\mathrm{div}(\mathbf{B})=0$ gives the usual incompressible Helmholtz--Hodge interpretation. Under the present $A=id$ assumption, Theorem \ref{propstrohelm} shows that the two conditions
\[
\mathrm{div}_{dx}(\mathbf{B})=0, \qquad \langle \nabla\Phi,\mathbf{B}\rangle=0
\]
also imply $\mathrm{div}_{\mu}(\mathbf{B})=0$, where $\mu=e^{2\Phi}dx$. This is the link between the deterministic Lyapunov interpretation and infinitesimal invariance of $\mu$ for the associated diffusion operator. The divergence-free condition also prevents the decomposition from degenerating in a trivial way. Indeed, if one tried to take $\nabla\Phi=0$, then $\mathbf{B}=\mathbf{G}$, and the HHD condition would require $\mathrm{div}(\mathbf{G})=0$. Thus, except for divergence-free vector fields, the HHD/SOHHD conditions force a nontrivial gradient contribution.
\end{remark}

\begin{theorem}\label{nonexofOHHD}
Let $\widetilde{\mathbf{G}}\in L^p_{loc}(\mathbb{R}^d, \mathbb{R}^d)$ for some $p\in(d,\infty)$. Suppose that $(L^{\widetilde{\mathbf{G}}}, C_0^{\infty}(\mathbb{R}^d))$ admits a unique infinitesimally invariant measure $\mu$. (Here unique is meant up to a multiplicative constant.) Write $\mu=\exp(2 \Phi)\,dx$, with $\Phi \in H^{1,p}_{loc}(\R^d)$ (cf. Remark \ref{definvaranrem}), and set $\mathbf{B}:=\widetilde{\mathbf{G}}-\nabla \Phi$. Then it holds:
\begin{itemize}
\item[(i)] $\widetilde{\mathbf{G}}=\nabla \Phi+\mathbf{B}$ is the unique WHHD for $\widetilde{\mathbf{G}}$.
  \item[(ii)] If there exists an OHHD for $\widetilde{\mathbf{G}}$, then it is unique and 
 \[
 \widetilde{\mathbf{G}}=\nabla \Phi+\mathbf{B}
 \]
is the unique OHHD for $\widetilde{\mathbf{G}}$.
\item[(iii)] An OHHD for $\widetilde{\mathbf{G}}$ exists, if and only if $\mathrm{div}_{dx}(\mathbf{B})=0$ or $\langle \nabla \Phi, \mathbf{B} \rangle = 0$.\\
(In particular, if $\langle \nabla \Phi, \mathbf{B} \rangle \not =0$ on $V$ for some measurable set $V$ with $dx(V)>0$ or if $\mathrm{div}_{dx }(\mathbf{B})\not =0$, then there is no OHHD for $\widetilde{\mathbf{G}}$.)
\end{itemize}
\end{theorem}
\begin{proof}
(i) By assumption $\nabla \Phi+\mathbf{B}$ is a WHHD for $\widetilde{\mathbf{G}}$. Suppose $\nabla \widetilde{\Phi}+\widetilde{\mathbf{B}}$ is another one, where $\tilde \Phi \in H^{1,2}_{loc}(\mathbb{R}^d)\cap L^{\infty}_{loc}(\mathbb{R}^d)$. Then $\exp(2\tilde \Phi) \, dx$ is infinitesimally invariant for $(L^{\widetilde{\mathbf{G}}}, C_0^{\infty}(\mathbb{R}^d))$ by Theorem \ref{propstrohelm}(i), and by Remark \ref{definvaranrem}  we obtain that $\exp(2 \tilde\Phi) \in H^{1,p}_{loc}(\mathbb{R}^d)$, hence $\tilde \Phi \in H^{1,p}_{loc}(\R^d)$. By the uniqueness of the infinitesimally invariant measure for $(L^{\widetilde{\mathbf{G}}}, C_0^{\infty}(\mathbb{R}^d))$, we must have $e^{2\tilde \Phi}=c_0 e^{2 \Phi}$ for some $c_0>0$. Then 
$$
\nabla \tilde \Phi =\frac{\nabla(e^{2\tilde \Phi})}{2e^{2\tilde \Phi}}=\frac{\nabla(c_0e^{2 \Phi})}{2c_0e^{2 \Phi}}=\nabla \Phi,
$$
and therefore $\widetilde{\mathbf{G}}=\nabla \Phi + \mathbf{B}$ is the unique WHHD for $\widetilde{\mathbf{G}}$.\\
(ii) Every OHHD is in particular a WHHD by Theorem \ref{propstrohelm}. Therefore the uniqueness in (ii) follows from the uniqueness in (i).\\
(iii) Suppose there exists an OHHD for $\widetilde{\mathbf{G}}$. Then by (i), $\widetilde{\mathbf{G}}=\nabla \Phi + \mathbf{B}$ is the unique OHHD for $\widetilde{\mathbf{G}}$. Hence $\langle \nabla \Phi, \mathbf{B} \rangle =0$, as well as $\mathrm{div}_{dx}( \mathbf{B})=0$. Conversely, if $\mathrm{div}_{dx}(\mathbf{B})=0$ or $\langle \nabla \Phi, \mathbf{B} \rangle = 0$, then since by Theorem \ref{propstrohelm}(i), $\mathrm{div}_{\mu}(\mathbf{B})=0$, thus $\widetilde{\mathbf{G}}=\nabla \Phi+\mathbf{B}$ is a WHHD,  we obtain from the last paragraph of Theorem \ref{propstrohelm} that $\widetilde{\mathbf{G}}=\nabla \Phi + \mathbf{B}$ is an OHHD.\\
\end{proof}

\begin{remark}\label{rem:SOHHD_exuniq}(Uniqueness of WHHDs)
The infinitesimally invariant measure $\mu$ in Theorem \ref{nonexofOHHD} is for instance unique, if the semigroup $(T_t)_{t>0}$ associated to the closed extension of $(L^{\widetilde{\mathbf{G}}},C_0^{\infty}(\mathbb{R}^d))$, or its associated Markov process, both constructed in \cite{LST22}, is recurrent (cf.  \cite[Theorem 3.15]{LT22} and \cite[Theorem 3.11, and Sections 3.2.2, 3.2.3]{LST22}). 
A sufficient condition for uniqueness of the WHHD $\widetilde{\mathbf{G}}=\nabla \Phi+\mathbf{B}$ and of the OHHD, if it exists, can therefore be stated purely in terms of $\Phi$ and $\mathbf{B}$. According to \cite[Proposition 3.4.2]{LST22}, or equivalently \cite[Theorem 21]{GT2} for the original result, it is given as follows. For $r>0$, let $B_r$ be the Euclidean ball of radius $r$ centered at the origin, and set
\begin{equation*}
v_1(r):=\mu(B_r), \quad v_2(r):=\int_{B_r}   \big | \big \langle \mathbf{B}(x),x \big \rangle \big | \mu(dx), \quad v(r):=v_1(r)+v_2(r),
\end{equation*}
and
$$
a_n:=\int_1^n \frac{r}{v(r)}dr, \ \ n\ge 1.
$$
Then, uniqueness of the WHHD $\widetilde{\mathbf{G}}=\nabla \Phi+\mathbf{B}$ holds, if
$$
\lim_{n\rightarrow \infty}a_n=\infty \ \ \ \text{and} \ \ \  \lim_{n\rightarrow \infty} \frac{\ln(v_2(n)\vee 1)}{a_n}=0.
$$
The latter is according to \cite[Corollary 22]{GT2} satisfied, if for large $r$ either of the following two conditions holds:
\begin{itemize}
\item[(a)] $v_1(r) \leq b r^2$ and $v_2(r) \leq b\log r$ for some constant $b>0$,
\item[(b)] $v(r) \leq c r^\alpha$ for some constants $c>0$ and $ \alpha <2$.
\end{itemize}
\end{remark}
The following counterexample shows that the decomposition \eqref{gradientVFdecomp} exploited in \cite{LT19} might be a unique gradient decomposition but not an OHHD, even in cases where $\Phi$ is explicitly given and smooth.
\begin{example}\label{counterexamplestrong}(No OHHD exists, but a unique WHHD exists and stability holds)
Define for $x=(x_1,\ldots, x_d)\in \R^d$
\[
C(x) = (c_{ij}(x))_{1 \leq i,j \leq d}, \quad 
c_{ij}(x) = 
\begin{cases}
x_1 & \text{if } (i,j)=(1,d) \\
-x_1 & \text{if } (i,j)=(d,1), \\
0 & \text{otherwise}.
\end{cases}
\]
Let $\mathbf{G}(x) =- x+  C(x)x + (0, \ldots,0, \frac{1}{2})$, where $C(x)x= (x_1x_d,0,\ldots,0,-x_1^2)$. 
Then, $- x+  C(x)x + (0, \ldots,0, \frac{1}{2})$  is the unique WHHD of $\mathbf{G}$, but $\mathbf{G}$ does not admit an OHHD, hence also no SOHHD.                                                              
\end{example}
\begin{proof}
Let $\mu=\rho\,dx$, where $\rho(x) = e^{2\Phi(x)}$, and  $\Phi:=-\frac12\|x\|^2$. Then $\beta^{\rho, id}(x)=\nabla \Phi(x)=-x$, and 
\begin{align*}
\beta^{\rho, C^T}(x)&= \frac{1}{2\rho} C^T(x) \nabla \rho(x) + \frac{1}{2} \operatorname{div} C(x) = C(x)x + (0,  \ldots, 0,\frac{1}{2}) 
\end{align*}
Thus $\mathbf{G}(x)=\nabla \Phi(x)+\beta^{\rho, C^T}(x)$, and since $\mathrm{div}_{\mu}(\beta^{\rho, C^T})=0$, we have that $\mathbf{G}$ admits a WHHD. In particular, by Theorem \ref{propstrohelm}(i), $\mu$ is an infinitesimally invariant measure for $(L^{\mathbf{G}},C_0^{\infty}(\mathbb{R}^d))$. Since
\begin{eqnarray*}
\langle \mathbf{G}(x),x \rangle =  -(x_1^2+\ldots +x_d^2)+\frac{1}{2} x_d,
\end{eqnarray*}
we obtain by \cite[Corollary 3.27]{LST22} that $(T_t^{\mu})_{t>0}$ is conservative, and since additionally $\mu$ is finite, we get by \cite[Lemma 3.18(ii)]{LT22}, that $\mu$ is the unique infinitesimally invariant measure for $(L^{\mathbf{G}},C_0^{\infty}(\mathbb{R}^d))$ up to a multiplicative constant. (Note that additionally by \cite[Lemma 3.18]{LT22} recurrence, i.e. stability, holds, and that $\mu$ is up to normalization the unique invariant measure.)
Furthermore
\begin{align*}
\langle \nabla \Phi, \beta^{\rho, C^T} \rangle(x) =  \Big \langle -x, C(x)x + (0, \ldots,  0,\frac{1}{2})  \Big \rangle = -\frac{1}{2} x_d,
\end{align*}
and $\mathrm{div}(\beta^{\rho, C^T})(x)=\mathrm{div}\big ((x_1x_d,0,\ldots,0,-x_1^2)+(0, \ldots,  0,\frac{1}{2})\big )=x_d$. Therefore, all assertions now follow from Theorem  \ref{nonexofOHHD}.
\end{proof}
\begin{example}\label{counterexamplestricktunique}(Two SOHHDs exist and no stability holds)
A simple example of nonuniqueness of a SOHHD is given by a constant vector field $\mathbf{G}\equiv (c_1,\ldots,c_d)\in \R^d$. Then $\Phi_1\equiv 0$ and $\Phi_2(x)=\langle (c_1,\ldots,c_d), x \rangle$, $x\in \R^d$ are two potential functions, i.e. $\mathbf{G}=0+(c_1,\ldots,c_d)$ is an SOHHD of $\mathbf{G}$ corresponding to $\Phi_1\equiv 0$, and $\mathbf{G}=(c_1,\ldots,c_d)+0$ is an SOHHD of $\mathbf{G}$ corresponding to $\Phi_2(x)=\langle (c_1,\ldots,c_d), x \rangle$. Another example (cf. Example \ref{exam:nonunique}) is given as follows: let
\[
\mathbf{G}=G=\begin{pmatrix}a&b\\ b&-a\end{pmatrix}, \quad a,b\in \R.
\]
Thus $\mathbf{G}\big ((x,y)\big )= (ax+by, bx-ay)$ and $\mathrm{div}(\mathbf{G})=\mathrm{trace}(G)=0$. Then, one can easily see that $\mathbf{G}=0+\mathbf{G}$ is an SOHHD of $\mathbf{G}$ with potential function $\Phi_1\equiv 0$, and $\mathbf{G}=\mathbf{G}+0$ is an SOHHD of $\mathbf{G}$ with potential function $\Phi_2\big ((x,y)\big )=\frac12(ax^2 +2bxy -ay^2)$.

\end{example}
Examples for existence and uniqueness of SOHHDs will be considered in the next sections.
\subsection{Linear vector fields}
Let $G=(g_{ij})_{1\le i,j\le d}$ be a matrix of real numbers. In this section, we will study linear vector fields 
$$
\mathbf{G}(x):= Gx, \qquad x\in \R^d,
$$
i.e. we consider $L^{A,\mathbf{G}}=L^{A,G}=L^{A,Gx}$ as in \eqref{thegenerator}. 
Thus
\begin{eqnarray}\label{LAG}
L^{A,G}f&=&\frac12\mathrm{trace}(A\nabla^2f)+\langle Gx,\nabla f\rangle \nonumber\\
&=& \frac12\sum_{i,j=1}^d a_{ij}\partial_{ij}f+\sum_{i=1}^d\big (\sum_{j=1}^d g_{ij}x_j\big ) \partial_i f,\quad f \in C_0^{\infty}(\mathbb{R}^d).
\end{eqnarray}
Let  $S=(s_{ij})_{1\le i,j\le d}$ be an arbitrary symmetric matrix of real numbers
\begin{eqnarray*}\label{linear2}
\Phi(x):=\frac12 \langle x, Sx\rangle, \qquad x\in \R^d,\qquad \mu:= e^{2\Phi(x)}dx=e^{\langle x, Sx\rangle}dx.
\end{eqnarray*}
Then, noting that $\nabla \Phi (x)=Sx$, we have that
\begin{eqnarray*}\label{L0A}
L^{A, A\nabla \Phi}f&=&\frac12\mathrm{trace}(A\nabla^2f)+\langle ASx,\nabla f\rangle, \quad f \in C_0^{\infty}(\mathbb{R}^d),
\end{eqnarray*}
is a symmetric operator with respect to $\mu$, and by Theorem \ref{propstrohelm}(i), $\mu$ is an infinitesimally invariant measure for $(L^{A,G},C_0^{\infty}(\mathbb{R}^d))$, if and only if $\mathrm{div}_{\mu}(Hx )=0$, where $Hx:=(G-AS)x$, $x\in \R^d$, $H=(h_{ij})_{1\le i,j\le d}$.
Since $\mathbf{B}=Hx$ is continuously differentiable, we know from Theorem \ref{propstrohelm}(ii) that  $\mathrm{div}_{\mu}(Hx)=0$ is equivalent to
\begin{eqnarray}\label{SOHHDlin1}
\langle Sx,Hx\rangle+\frac12 \mathrm{div}(Hx)=0, \qquad x\in \R^d.
\end{eqnarray}
We can rewrite \eqref{SOHHDlin1} as
\begin{eqnarray}\label{SOHHDlin2}
&&\sum_{i=1}^{d}\big (S x\big )_i\big (H x\big )_i+\frac12 \sum_{i=1}^{d}\partial_i \big (\sum_{j=1}^d h_{ij}x_j\big )\nonumber \\
&=&\sum_{i=1}^{d}\big (\sum_{j=1}^{d}s_{ij} x_j\big )\big (\sum_{k=1}^{d}h_{ik}x_k\big )+\frac12 \sum_{i=1}^{d} h_{ii}
=0, \qquad x\in \R^d.
\end{eqnarray}
Choosing $x=0$, we obtain
\begin{eqnarray}\label{SOHHDlin3}
\sum_{i=1}^{d} h_{ii}=0.
\end{eqnarray}
On the basis of these preliminary considerations, we now deduce the following lemma.
\begin{lemma}\label{lemma algebraic riccati}
Let  $S=(s_{ij})_{1\le i,j\le d}$, $G=(g_{ij})_{1\le i,j\le d}$, be matrices of real numbers, and $S$ be symmetric. Then $\mu= e^{2\Phi(x)}dx=e^{\langle x, Sx\rangle}dx$, with
$\Phi(x)=\frac12 \langle x, Sx\rangle$, is an infinitesimally invariant measure for $(L^{A,G},C_0^{\infty}(\mathbb{R}^d))$ defined by \eqref{LAG}, if and only if 
\begin{eqnarray}\label{SOHHDlinalg1}
SG+G^TS=2SAS\quad \text{and}\quad \mathrm{trace}(G-AS)=0. 
\end{eqnarray}
Moreover, \eqref{SOHHDlinalg1} is equivalent to the following system of equations
\begin{equation}\label{SOHHDlin5}
\begin{cases}
\ \sum_{i=1}^{d}s_{ij}h_{ij} =0, \qquad 1\le j \le d, \\[6pt]
\  \sum_{i=1}^{d}(s_{ij}h_{ik} +s_{ik}h_{ij})= 0, \qquad (j,k)\in \{1,\ldots, d\}^2, \  j <k, \\[6pt]
\ \sum_{i=1}^{d} h_{ii}=0,
\end{cases}
\end{equation}
where $H=(h_{ij})_{1\le i,j\le d}=G-AS$. In particular, 
\begin{eqnarray}\label{SOHHDfirst}
\tilde{G}x:=Sx+(G-AS)x, \qquad x\in \R^d,
\end{eqnarray}
is a WHHD of $\tilde{G}x$, if and only if it is a SOHHD of $\tilde{G}x$, if and only if $\mu=e^{\langle x, Sx\rangle}\text{d}x$ is an infinitesimally invariant measure for $(L^{A,G},C_0^{\infty}(\mathbb{R}^d))$.
\end{lemma}
\begin{proof}
By \eqref{SOHHDlin3} and \eqref{SOHHDlin1}, we can see that $\mu=e^{\langle x, Sx\rangle}\text{d}x$ is an infinitesimally invariant measure for $(L^{A,G},C_0^{\infty}(\mathbb{R}^d))$, if and only if 
\begin{eqnarray*}
\langle x,SHx\rangle=0, \quad x\in \R^d\quad \text{and}\quad \frac12 \mathrm{trace}(H)=0. 
\end{eqnarray*}
$\langle x,SHx\rangle=0$ for all $x\in \R^d$ means that the symmetric part of $SH$ is zero, i.e. $SH+(SH)^T=SH+H^T S=0$. Substituting $H=G-AS$, we obtain the first statement. The equivalence of \eqref{SOHHDlin5} and \eqref{SOHHDlinalg1} follows, since \eqref{SOHHDlin2} is equivalent to \eqref{SOHHDlin3} and 
\begin{eqnarray*}\label{SOHHDlin4}
\sum_{i=1}^{d}\big (\sum_{j,k=1}^{d}s_{ij}h_{ik} x_jx_k\big )=\sum_{j=1}^{d}\big (\sum_{i=1}^{d}s_{ij}h_{ij} x_j^2\big )+\sum_{j<k}\big (\sum_{i=1}^{d}(s_{ij}h_{ik} +s_{ik}h_{ij})x_j x_k\big )=  0
\end{eqnarray*}
for any $x=(x_1,\ldots,x_d)\in \R^d$, where $\sum_{j<k}$ stands for the sum over all pairs $(j,k)\in \{1,\ldots, d\}^2$ with $j<k$.\\
Consider $\tilde{G}:=S+H(=\nabla \Phi +H)$. Using \eqref{SOHHDlin3} in \eqref{SOHHDlin1}, we see that 
\begin{eqnarray*}
\mathrm{div}_{\mu}(Hx)=0\quad \Longleftrightarrow \quad \forall x\in \R^d:\quad \langle Sx,Hx\rangle=0 \quad \text{and}\quad \frac12 \mathrm{div}(Hx)=0,
\end{eqnarray*}
which implies the equivalence of WHHD and SOHHD by checking the conditions of Definition \ref{helmdefn}. The remaining equivalence related to the infinitesimal invariance follows from Theorem \ref{propstrohelm}(i) applied to $\mathbf{B}=Hx$. 
\end{proof}
\subsubsection{Closed form solution for the case $d=2$}
As we will see in Theorem \ref{sol:algRicgen}, one can find a solution $S$ to \eqref{SOHHDlinalg1} using the theory of algebraic Riccati equations and the continuous time Lyapunov equation.  Before doing that, we will discuss shortly the direct approach via \eqref {SOHHDlin5}. The first two lines of \eqref{SOHHDlin5} consist of $\frac{d(d+1)}{2}$ equations and have exactly the same number of  free variables $s_{ij}$, $i\le j$. It should therefore be solvable in general. However, \eqref{SOHHDlin5} consists of quadratic equations with side condition $\mathrm{trace}(H)=0$ and seems to be difficult to solve by hand by elimination of variables. For $d=2$, one can derive a closed form solution of \eqref{SOHHDlin5} by direct computations (see \cite[Proposition 1]{Sud2} for the case $A=id$, and below for $A\not= id$). For $d\ge 3$ it is unclear to us whether such a closed form solution exists.

\begin{lemma}\label{allG}
Consider a general  $2\times 2$ matrix $G$ of real numbers, i.e.
\[
G=\begin{pmatrix}a&b\\ c&d\end{pmatrix},\qquad
\quad a,b,c,d\in \R.
\]
Then $G$ satisfies exactly one of the following conditions:
\begin{itemize}
\item[(i)] $\mathrm{trace}(G)=a+d=0$.
\item[(ii)] $\mathrm{trace}(G)=a+d \neq 0$ and $G$ singular, i.e. $\det(G)=ad-bc=0$
\item[(iii)] $\mathrm{trace}(G)=a+d \neq 0$ and $G$ non-singular, i.e. $\det(G)=ad-bc\not=0$.
\end{itemize}
In particular, in the present case of $d=2$, $(iii)$ is equivalent to the condition of Theorem \ref{mainlapuin}.
\end{lemma}
\begin{proof}
That $G$ satisfies exactly one of the conditions $(i)-(iii)$ is clear. So, we only have to show the equivalence of condition $(iii)$ and the condition of Theorem \ref{mainlapuin}. For this let $\sigma(G) = \{ \lambda_1,\lambda_2 \} \subset \mathbb{C}$ be the set of eigenvalues of $G$ counting multiplicities. Since $\mathrm{trace}(G)=\lambda_1 + \lambda_2$, the trace condition $\mathrm{trace}(G)=a+d \neq 0$ in $(iii)$ is equivalent to $\lambda_1 + \lambda_2 \neq 0$. Moreover, $G$ non-singular is equivalent to $\lambda_1, \lambda_2 \not=0$, which is further equivalent to $\lambda_1+\lambda_1=2\lambda_1\not=0$ and $\lambda_2+\lambda_2=2\lambda_2\not=0$. 
\end{proof}
\begin{theorem} \label{allGOHHD}
Consider two $2\times 2$ matrices $G, A$ of real numbers, where $A$ is symmetric and positive definite, i.e.
\[
G=\begin{pmatrix}a&b\\ c&d\end{pmatrix},\qquad
A=\begin{pmatrix}\alpha&\beta\\ \beta&\gamma\end{pmatrix}
\quad(\alpha>0,\ \delta:=\det(A)=\alpha\gamma-\beta^{2}>0).
\]
Then, there exists a symmetric matrix $S=\begin{pmatrix}s_{11}& s_{12}\\ s_{12}& s_{22}\end{pmatrix}$ solving \eqref{SOHHDlinalg1}, and which is of the following form:
\begin{itemize}
\item[(i)]
If $\mathrm{trace}(G)=a+d=0$, we can choose $S=0$ as symmetric solution.
\item[(ii)] If $\mathrm{trace}(G)=a+d \neq 0$ and $G$ singular, i.e. $\det(G)=ad-bc=0$, we can choose
\begin{eqnarray*}
S&=&\frac{a+d}{\alpha a^2+2\beta ab + \gamma b^2}\begin{pmatrix}a^2& ab\\ ab& b^2\end{pmatrix}, \qquad \text{if }\ b\not =0,\\[3pt]
S&=&\frac{a+d}{\alpha c^2+2\beta cd + \gamma d^2}\begin{pmatrix}c^2& cd\\ cd& d^2\end{pmatrix}, \qquad \text{if }\ b=0,\ c\not=0,
\end{eqnarray*}
and in case $b=c=0$, we may choose 
\begin{eqnarray*}
S\ =\ \begin{pmatrix}\frac{a}{\alpha}& 0\\ 0& 0\end{pmatrix}, \qquad \text{if }\ a\not =0,\qquad
S\ =\ \begin{pmatrix}0& 0\\ 0& \frac{d}{\gamma}\end{pmatrix}, \qquad \text{if }\ d\not=0.
\end{eqnarray*}
\item[(iii)] If $\mathrm{trace}(G)=a+d \neq 0$  and $G$ non-singular, i.e. $\det(G)=ad-bc\not=0$, we can choose
\begin{eqnarray*}
S\ =\ \kappa\begin{pmatrix}\gamma a(a+d) + \alpha c^2 - \gamma bc - 2\beta ac& \alpha cd + \gamma ab - 2\beta ad\\ \alpha cd + \gamma ab - 2\beta ad& \alpha d(a+d) + \gamma b^2 - \alpha bc - 2\beta bd\end{pmatrix}, 
\end{eqnarray*}
where
\[
\kappa=\frac{a+d}{ (\alpha \gamma-\beta^2) (a+d)^2 +\big( \gamma b - \alpha c + \beta(a-d)\big)^2}.
\]
\end{itemize}
\end{theorem}
\begin{proof}
(i) If $\mathrm{trace}(G)=0$, we can choose $S=0$. Then, $S$ is obviously symmetric and satisfies \eqref{SOHHDlinalg1}.\\
(ii) $a+d \neq 0$ and $ad=bc$, then $G$ is a rank-one matrix and can be written as outer product of two vectors $u$ and $v^T$. For these vectors the matrix notation is used. There are four different cases:
\begin{eqnarray*}
G&=& u v^T\ =\ \begin{pmatrix} 1 \\ \frac{d}{b} \end{pmatrix} \begin{pmatrix} a & b \end{pmatrix}\ = \ \begin{pmatrix}a& b\\ \frac{ad}{b}& d\end{pmatrix}, \quad \text{if }\ b\not =0,
\end{eqnarray*}
\begin{eqnarray*}
G&=&uv^T\ =\ 
\begin{pmatrix} \frac{a}{c} \\ 1 \end{pmatrix} \begin{pmatrix} c & d \end{pmatrix}
\ = \ \begin{pmatrix}a& \frac{ad}{c}\\ c& d\end{pmatrix}, \quad \text{if }\ b=0,\ c\not=0,
\end{eqnarray*}
and if $b=c=0$
\begin{eqnarray*}
G&=&uv^T\ =\ 
\begin{pmatrix} 1 \\ 0 \end{pmatrix} \begin{pmatrix} a& 0 \end{pmatrix}
\ = \ \begin{pmatrix}a& 0\\ 0& 0\end{pmatrix}, \quad \text{if }\ d=0,
\end{eqnarray*}
and
\begin{eqnarray*}
G&=&uv^T\  =\ 
\begin{pmatrix} 0 \\ 1 \end{pmatrix}\begin{pmatrix} 0& d\end{pmatrix}
\ = \ \begin{pmatrix}0& 0\\ 0& d\end{pmatrix}, \quad \text{if }\ a=0.
\end{eqnarray*}
For such matrices $G$ one can choose the ansatz $S=\kappa_v vv^T$ for a solution, where $\kappa_v$ is a constant depending on $v$. Then, noting that $v^T u=\mathrm{trace}(G)=a+d$ and that $u^T  v=\mathrm{trace}(G^T)=a+d$, we get
\begin{eqnarray*}
SG+G^{T}S\ =\ \kappa_v vv^T u v^T+(u v^T)^T\kappa_v vv^T\ =\ \kappa_v v(v^T u)v^T+\kappa_v v (u^T  v)v^T\ = \ 2\kappa_v (a+d) vv^T,
\end{eqnarray*}
and
\begin{eqnarray*}
2SAS \ =\ 2 \kappa_v  vv^T A \kappa_v  vv^T \ =\ 2 \kappa_v^2  v(v^T A v)v^T\ = \ 2 \kappa_v^2 \langle v,Av\rangle vv^T.
\end{eqnarray*}
thus in order to have $SG+G^{T}S=2SAS$, we need to fix $\kappa_v$ as
\[
\kappa_v=\frac{a+d}{\langle v,Av\rangle}.
\]
Moreover, it then also holds that 
\begin{eqnarray*}
\mathrm{trace}(AS)=\kappa_v \,\mathrm{trace}(Avv^T)=\kappa_v \,\mathrm{trace}(v^TAv)=\kappa_v \,\mathrm{trace}(\langle v,Av\rangle)=a+d=\mathrm{trace}(G).
\end{eqnarray*}
Collecting the results, we obtain the assertion, namely
\begin{equation*}\label{}
S=\begin{cases}
\medskip
\ \kappa_v \begin{pmatrix} a \\ b \end{pmatrix}\begin{pmatrix} a & b \end{pmatrix}\ = \ \frac{a+d}{\alpha a^2+2\beta ab + \gamma b^2} \begin{pmatrix}a^2& ab\\ ab& b^2\end{pmatrix}, \quad \text{if }\ b\not =0,\\ 
\medskip
\ \kappa_v \begin{pmatrix} c \\ d \end{pmatrix}\begin{pmatrix} c & d \end{pmatrix}\ = \ \frac{a+d}{\alpha c^2+2\beta cd + \gamma d^2} \begin{pmatrix}c^2& cd\\ cd& d^2\end{pmatrix}, \quad \text{if }\ b=0,\ c\not=0,\\
\medskip
\ \kappa_v \begin{pmatrix} a \\ 0 \end{pmatrix}\begin{pmatrix} a & 0 \end{pmatrix}\ = \ \frac{a}{\alpha a^2}  \begin{pmatrix}a^2& 0\\ 0& 0\end{pmatrix}, \quad \text{if }\ d=0,\\
\medskip
\ \kappa_v \begin{pmatrix} 0 \\ d \end{pmatrix} \begin{pmatrix} 0 & d \end{pmatrix}\ = \ \frac{d}{\gamma d^2} \begin{pmatrix} 0& 0\\ 0& d^2\end{pmatrix}, \quad \text{if }\ a=0.
  \end{cases}
\end{equation*}
(iii) In case $\mathrm{trace}(G)=a+d \neq 0$ and $G$ non-singular, we know by Lemma \ref{allG}  that the conditions of Theorems \ref{mainlapuin} and \ref{explicitconstruction} are met. Therefore applying these theorems (for the notation used below, see Section \ref{section:algRiceq}), we obtain that
$$
S=(-P)^{-1}, \qquad \text{where }\text{\rm Vec}(P) = - (\mathcal{K})^{-1} \text{\rm Vec}(2A), 
$$
with
\[\mathcal{K}=I\otimes G
+G\otimes I=
\begin{pmatrix}
a&b&0&0\\
c&d&0&0\\
0&0&a&b\\
0&0&c&d\\
\end{pmatrix}+\begin{pmatrix}
a&0&b&0\\
0&a&0&b\\
c&0&d&0\\
0&c&0&d\\
\end{pmatrix}=\begin{pmatrix}
2a&b&b&0\\
c&a+d&0&b\\
c&0&a+d&b\\
0&c&c&2d\\
\end{pmatrix}.
\]
Using standard inversion formulas for the $4\times 4$  matrix $\mathcal{K}$ and the $2\times 2$ matrix $P$, we can readily check that the assertion holds.
\end{proof}\\
Summarizing, we get the following proposition which extends \cite[Proposition 1]{Sud2} to the case $A\not=id$.
\begin{proposition}\label{linearSOHHD}
Let $\mathbf{G}(x):= Gx$, $x\in \R^2$, be a linear vector field represented by a $2\times 2$ matrix $G$ of real numbers, $\Phi(x):=\frac12 \langle x, Sx\rangle$, $x\in \R^2$,
where $A,S$ are defined as in Theorem \ref{allGOHHD}, $\mu:= e^{2\Phi(x)}\text{d}x$. Then 
$$
\tilde Gx=Sx+(G-AS)x, \quad x\in \R^2
$$ 
is a SOHHD and $\mu$ is an infinitesimally invariant measure for $\big (\frac{1}{2} \mathrm{trace}(A\nabla^2)+ \langle \mathbf{G}, \nabla \cdot \rangle, C_0^{\infty}(\R^2)\big )$.\\
\end{proposition}
\subsubsection{The general case $d\ge 2$}
Proposition \ref{propintegrep} below provides in finite dimensions an explicit characterization of infinitesimally invariant measures for the Ornstein-Uhlenbeck operator and invariant measures for the Ornstein-Uhlenbeck semigroup. 
In particular, by using the explicit form of the Lyapunov equation and the algebraic Riccati equation, Proposition \ref{propintegrep}(ii) represents an alternative proof of the approach in \cite[Proposition 9.3.1]{LB07} or \cite[Section 3]{LMP20}. Moreover, Proposition \ref{propintegrep}(iii) provides a simple alternative proof of the result presented in \cite[Section 11.2.3]{DZ92} (see also \cite[Remark 9.3.2]{LB07}).

\begin{proposition}\label{propintegrep}
Let $(L^{A,G},C_0^{\infty}(\mathbb{R}^d))$ be given as in \eqref{LAG}, where $A$ is symmetric and positive definite and $G$ is arbitrary. Then it holds:
\begin{itemize}
\item[(i)] There exists a symmetric matrix of real numbers $S=(s_{ij})_{1\le i,j\le d}$, such that $\mu=e^{\langle x, Sx\rangle}\text{d}x$ is an infinitesimally invariant measure for $(L^{A,G},C_0^{\infty}(\mathbb{R}^d))$.
\item[(ii)] Assume that $G$ is a Hurwitz matrix, i.e. all eigenvalues of $G$ have strictly negative real parts. Let
$$
S=\left(-\int_{0}^{\infty} e^{Gt}\,(2A)\,e^{G^T t}\,dt \right)^{-1}.
$$
Then $S$ is negative definite and $\mu=e^{\langle x, Sx\rangle}\text{d}x$ is the unique infinitesimally invariant measure for $(L^{A,G},C_0^{\infty}(\mathbb{R}^d))$ and the unique (finite) invariant measure for the semigroup associated to the closed extension of $(L^{A,G},C_0^{\infty}(\mathbb{R}^d))$ and its associated stochastic process, both constructed in \cite{LST22}. (Here uniqueness is meant up to a multiplicative constant.)\\
In particular, let $\tilde{G}:= S+(G-AS)$. Then, $\tilde{G}$ is Hurwitz and
$$
Sx+(G-AS)x
$$
is the unique SOHHD of $\tilde Gx$, and $(L^{A,G},C_0^{\infty}(\mathbb{R}^d))$ and $(L^{\widetilde{G}},C_0^{\infty}(\mathbb{R}^d))$ have the same unique (finite) infinitesimally invariant measure $\mu$.
\item[(iii)] Suppose that the semigroup $(T_t)_{t>0}$ associated to the closed extension of $(L^{A,G},C_0^{\infty}(\mathbb{R}^d))$ or its associated stochastic process, both constructed in \cite{LST22}, admit a finite invariant measure. Then $G$ is a Hurwitz matrix. Thus, the OU-semigroup $(T_t)_{t>0}$ or the OU-process admit a finite invariant measure if and only if $G$ is a Hurwitz matrix.
\end{itemize}
\end{proposition}
\begin{proof}
(i) By Lemma \ref{lemma algebraic riccati} the infinitesimal invariance of $\mu$ is equivalent to $S$ being a solution of \eqref{SOHHDlinalg1}, which follows from Theorem \ref{sol:algRicgen}.\\
(ii) By \cite[Theorem 4.6]{K02}, $P:=\int_{0}^{\infty} e^{Gt}\,(2A)\,e^{G^T t}\,dt$ is the unique positive definite solution of the Lyapunov equation $GP+PG^T=-2A$, where we used that together with $G$, $G^T$ is also Hurwitz. Thus $S:=(-P)^{-1}$ is negative definite, and upon multiplying the Lyapunov equation from the left with $S$ and from the right with $-S$, we can see that it is equivalent  to 
$SG+G^TS=2SAS$. The trace condition then holds automatically since $SG+G^TS=2SAS$ implies $G+S^{-1}G^TS=2AS$, hence 
\begin{align*} 
\mathrm{trace}(2AS) &= \mathrm{trace}\bigl(G+S^{-1}G^{T}S\bigr)= \mathrm{trace}(G)+\mathrm{trace}\bigl(S^{-1}G^{T}S\bigr)\\
&= \mathrm{trace}(G)+\mathrm{trace}\bigl(G^{T}SS^{-1}\bigr)= \mathrm{trace}(G)+\mathrm{trace}(G^{T})=2\,\mathrm{trace}(G).
\end{align*}
Thus by Lemma \ref{lemma algebraic riccati}, $\mu=e^{\langle x, Sx\rangle}\text{d}x$ is an infinitesimally invariant measure for $(L^{A,G},C_0^{\infty}(\mathbb{R}^d))$.
Moreover, it follows from the last part of  Lemma \ref{lemma algebraic riccati} that $\tilde Gx=Sx+(G-AS)x$ is a SOHHD. It remains to show the statements about uniqueness of the given SOHHD for $\tilde Gx$ and $\mu$. Since $Sx$, $(G-AS)x$, are continuously differentiable, the uniqueness of the SOHHD follows from the uniqueness of $\tilde Gx=Sx+(G-AS)x$ regarded as OHHD (cf. last part of  Definition \ref{helmdefn}). Since $S$ is negative definite, the measure
\[
\mu = e^{\langle x,Sx\rangle}\,dx
\]
is finite. By \cite[Proposition 3.13 and Theorem 3.15]{LT22} (both statements are applicable, since by the linearity of  $G$ conservativeness automatically holds), $\mu$ is the unique infinitesimally invariant measure for  $(L^{A,G},C_0^{\infty}(\mathbb{R}^d))$. It is an immediate consequence of  $SG+G^TS=2SAS$ that $\widetilde{G}$ satisfies the Lyapunov equation $S\widetilde G+\widetilde{G}^TS=2S^2$, and therefore by \cite[Theorem 4.6]{K02} $\widetilde{G}$ is Hurwitz. By what we have just shown for $L^{A,G}$, we then have that $\mu$ is the unique infinitesimally invariant measure for  $(L^{\widetilde{G}},C_0^{\infty}(\mathbb{R}^d))$. Therefore, by Theorem \ref{nonexofOHHD}(ii) (recall that $\Phi(x)=\frac12 \langle x, Sx\rangle$ and $\nabla \Phi (x)=Sx$), the decomposition $\tilde Gx=Sx+(G-AS)x$ is the unique OHHD of $\tilde Gx$. Finally, the uniqueness of $\mu$ as invariant measure follows from \cite[Theorem 3.17]{LT22}.\\
(iii) Suppose that the semigroup $(T_t)_{t>0}$ associated to the closed extension of $(L^{A,G},C_0^{\infty}(\mathbb{R}^d))$ or equivalently its associated stochastic process, both constructed in \cite{LST22}, admit a finite invariant measure $\mu$. Then by \cite[Proposition 3.13 and Theorem 3.15]{LT22}, $\mu$ is the unique infinitesimally invariant measure for  $(L^{A,G},C_0^{\infty}(\mathbb{R}^d))$. Consequently by (i), we must have that $\mu=e^{\langle x, Sx\rangle}\text{d}x$ for some symmetric matrix of real numbers $S=(s_{ij})_{1\le i,j\le d}$. Then $S$ is negative definite by finiteness of $\mu$ and a solution to \eqref{SOHHDlinalg1}. We hence conclude by \cite[Theorem 4.6]{K02} that $G$ must be Hurwitz.
\end{proof}

\section{Weighted Helmholtz--Hodge decompositions and nonlinear perturbations}\label{sec: WHHD and nonlinear}

In the previous section, we discussed the weighted and orthogonal Helmholtz--Hodge decompositions for a linear vector field $Gx$, where $G$ is a $d \times d$ matrix. In this section, we extend our analysis to cases where the linear vector field $Gx$ is perturbed by additional polynomial terms. Our main goal in this section is to determine the conditions on these perturbation terms that ensure the potential function $\Phi=\frac{1}{2} \langle Sx, x \rangle$, derived from the linear case, is robustly preserved, and moreover to present SOHHDs as well as unique explicit (infinitesimally) invariant measures for general classes of nonlinear drifts.

\subsection{Polynomial vector fields in $\R^d$}\label{sec:3.1}

In order to be precise by what we mean by polynomial vector fields, we start with the definition. A  homogeneous polynomial of degree $k\in  \N\cup\{0\}$ in $\R^d$ is an arbitrary non-zero linear combination of monomials of the form
\[
x_1^{i_1}x_2^{i_2}\ldots x_d^{i_d}, \qquad i_1+i_2+\ldots+i_d=k, \quad  i_1, i_2, \ldots, i_d\in \N\cup\{0\},\quad x=(x_1,\ldots ,x_d)\in \R^d.
\]
A polynomial $P$ on $\R^d$ is a sum 
\[
P=\sum_{k=0}^{n}P_k, \quad n\in  \N\cup\{0\},
\]
where for each $0\le k\le n$, $P_k$ is a homogeneous polynomial of degree $k$ in $\R^d$ or zero. If $P_n\not \equiv 0$ ($P_n$ is not the zero function), we say that $P$ has degree $n$. In that case, we call $\sum_{k=0}^{n}P_k$ its homogeneous decomposition, $P_k$ its homogeneous part of order $k\in\{0,\ldots,n\}$. A homogeneous decomposition is unique. Since we will not introduce more structure, we can also define the degree of the zero polynomial to be zero. 
A polynomial vector field $\mathbf{P}$ in $\R^d$ is a sum 
\[
\mathbf{P}=\sum_{k=0}^n\mathbf{P}_k
\]
where for each $0\le k\le n$, $\mathbf{P}_k=(P_{k,1},P_{k,2}, \ldots, P_{k,d})$ is a $d$-tuple of entries which are either a homogeneous polynomial of degree $k$ or zero. If $P_{n,j}\not \equiv 0$ for some $j\in \{1,\ldots,d\}$, we say that $\mathbf{P}$ has degree $n$.
\begin{theorem}\label{thm:polydifdrift}
Let $\Phi$ be a non-constant polynomial on $\R^d$ of even degree $m\ge 2$ and homogeneous decomposition $\Phi(x)=\sum_{k=0}^m \Phi_k(x)$. Let $\mu=e^{2 \Phi}\,dx$. Suppose that there exist constants $c\in\mathbb{R}$ and $\alpha>d/2$ such that the estimate
\begin{eqnarray}\label{prop:eq0}
 \Phi(x)\le c-\alpha\log\|x\|
\end{eqnarray}
holds outside some arbitrarily large compact set. Choose a constant $C>0$, such that $V:=-\Phi+C\ge 1$. Let $A=(a_{ij})_{1\le i,j\le d}$ be a symmetric and positive definite matrix of real numbers, and $\mathbf{P}$ a polynomial vector field in $\R^d$ of degree $n\ge 0$, such that 
\begin{eqnarray}\label{prop:eq1}
\langle \nabla\Phi, \mathbf{P}-A\nabla\Phi \rangle+\frac12\mathrm{div}(\mathbf{P}-A\nabla\Phi)=0.
\end{eqnarray}
Let $(P_t)_{t>0}$ be the regularized semigroup associated with the $L^1$-closed extension of $(L^{A,\mathbf{P}}, C_0^{\infty}(\mathbb{R}^d))$ (see \cite{LST22}), and suppose that there exist constants $N_0\in \N$ and $M>0$ such that
\begin{eqnarray}\label{prop:eq2}
L^{A,\mathbf{P}}V \leq M V \quad \text{on } \mathbb{R}^d \setminus \overline{B}_{N_0}.
\end{eqnarray}
Then $(P_t)_{t>0}$ is conservative (cf. \cite[Lemma 3.26]{LST22}) and $\mu$ is the unique finite invariant measure for $(P_t)_{t>0}$\ and the unique finite infinitesimally invariant measure for $(L^{A,\mathbf{P}}, C_0^{\infty}(\mathbb{R}^d))$ as well as a finite infinitesimally invariant measure for $(L^{\widetilde{\mathbf{P}}}, C_0^{\infty}(\mathbb{R}^d))$. If, in addition, the regularized semigroup associated with the
$L^1$-closed extension of $(L^{\widetilde{\mathbf P}},C^\infty_0(\mathbb R^d))$
is conservative, then $\mu$ is also the unique finite invariant
measure for this semigroup.
Moreover, if $\mathrm{div}(\mathbf{P}-A\nabla\Phi)=0$, then 
\[
\widetilde{\mathbf{P}}:=\nabla \Phi +\mathbf{B}, \qquad \mathbf{B}:=\mathbf{P}-A\nabla\Phi
\]
is the unique SOHHD for $\widetilde{\mathbf{P}}$.
\end{theorem}
\begin{proof} By the assumption \eqref{prop:eq0} on $\Phi$, there exists $R>0$ such that for all $x$ with $\|x\|\ge R$
\[
 e^{2\Phi(x)}
 \le e^{2c}\|x\|^{-2\alpha}\in   L^1(\mathbb{R}^d\setminus B_{R}(0)).
\]
Therefore $\mu$ is finite.
By Theorem \ref{propstrohelm}, the condition \eqref{prop:eq1} implies that $\mu$ is infinitesimally invariant for $(L^{A,\mathbf{P}}, C_0^{\infty}(\mathbb{R}^d))$ and $(L^{\widetilde{\mathbf{P}}}, C_0^{\infty}(\mathbb{R}^d))$. Since $\mu$ is finite and conservativeness holds by \eqref{prop:eq2}, or conservativeness of the regularized semigroup associated with the
$L^1$-closed extension of $(L^{\widetilde{\mathbf P}},C^\infty_0(\mathbb R^d))$ is assumed, the remaining assertions about $\mu$ follow by \cite[Proposition 3.13 and Theorem 3.15]{LT22}.
Finally, that $\widetilde{\mathbf{P}}:=\nabla \Phi +\mathbf{B}$ is an SOHHD under the assumption $\mathrm{div}(\mathbf{P}-A\nabla\Phi)=0$ follows from \eqref{prop:eq1} and the remark at the end of Definition \ref{helmdefn}, and the uniqueness of the SOHHD follows from Theorem \ref{nonexofOHHD}(ii).
\end{proof}

\begin{remark}\label{rem:polydifdrift}
In the situation of Theorem \ref{thm:polydifdrift}, it holds:\\
(i) For $V=-\Phi+C$, we have $\nabla V=-\nabla\Phi$ and $\nabla^2 V=-\nabla^2\Phi$. Using in particular \eqref{prop:eq1}, we obtain
\begin{eqnarray*}
L^{A,\mathbf{P}}V &=& \frac12\mathrm{trace}(A\nabla^2V)+\langle A\nabla\Phi+\mathbf{P}-A\nabla\Phi,\nabla V\rangle\\
&=& -\frac12\mathrm{trace}(A\nabla^2\Phi)-\langle  A\nabla\Phi, \nabla\Phi\rangle -\langle  \nabla\Phi, \mathbf{P}-A\nabla\Phi\rangle\\
&=& -\frac12\mathrm{trace}(A\nabla^2\Phi)-\langle  A\nabla\Phi, \nabla\Phi\rangle+\frac12\mathrm{div}(\mathbf{P}-A\nabla\Phi).
\end{eqnarray*}
Therefore, the non-explosion condition \eqref{prop:eq2} is exactly the same as in the reversible case, if $\widetilde{\mathbf{P}}:=\nabla \Phi+\mathbf{B}$, with $\mathbf{B}:=\mathbf{P}-A\nabla\Phi$, is a SOHHD, and in case $\widetilde{\mathbf{P}}:=\nabla \Phi+\mathbf{B}$ is not an OHHD, one can observe that $\frac12\mathrm{div}(\mathbf{P}-A\nabla\Phi)$ has one degree less than $\mathbf{B}$.\\
(ii) Let $\mathbf{B}:=\mathbf{P}-A\nabla\Phi$ and $\mathrm{deg}(\mathbf{B})=q=\max(n,m-1)$, and express both, $\Phi$ and $\mathbf{B}$, in their homogeneous decompositions
\[
\Phi=\sum_{k=0}^m \Phi_k,\qquad \mathbf{B}=\sum_{\ell=0}^q \mathbf{B}_{\ell},
\]
where $\Phi_k$ is homogeneous polynomial of degree $k$ and $\mathbf{B}_{\ell}$ is a homogeneous vector field of degree $\ell$. In particular $\Phi_0$ is constant and $\mathbf{B}_0$ is a constant vector field. Then, collecting the respective degrees, \eqref{prop:eq1} holds, if and only if for each homogeneous degree $s=0,1,\dots,m+q-1$,
\begin{equation*}
\sum_{k=1}^{\min(m,s+1)} \langle \nabla\Phi_k,  \mathbf{B}_{s+1-k}\rangle\ +\ \frac12\mathrm{div}(\mathbf{B}_{s+1})=0,
\end{equation*}
with the convention $\mathbf{B}_j\equiv 0$ for $j\notin\{0,1,\dots,q\}$.\\
(The term $\langle \nabla\Phi_k,\mathbf{B}_\ell\rangle$ is homogeneous of degree $(k-1)+\ell=k+\ell-1$ or zero, while $\mathrm{div}(\mathbf{B}_\ell)$ is homogeneous of degree $\ell-1$ or zero (for $\ell\ge 1$)). The converse is immediate by summing all homogeneous degrees.
\end{remark}

\begin{proposition}\label{prop:lyap-criterion}
Let $\Phi$, $A$, $\mathbf{P}$, $\mu$, $\mathbf{B}$, and $V$ be as in Theorem \ref{thm:polydifdrift} and let all conditions of Theorem \ref{thm:polydifdrift} except \eqref{prop:eq0} and \eqref{prop:eq2} be satisfied. Then a sufficient condition for \eqref{prop:eq0} and \eqref{prop:eq2} to hold is given by
\begin{itemize}
\item[(i)] $\mathrm{deg} (\mathbf{B})\le 2m-2$, where $m=\mathrm{deg}(\Phi)$, and
\item[(ii)] $-\Phi_m(x)>0$ for every $x\in S^{d-1}:=\{x\in \R^d\,| \, \|x\|=1\}$.
\end{itemize}
In particular, Theorem \ref{thm:polydifdrift} applies and in addition, the regularized semigroup associated with the
$L^1$-closed extension of $(L^{\widetilde{\mathbf P}},C^\infty_0(\mathbb R^d))$
is conservative.
\end{proposition}
\begin{proof}
Since  \(S^{d-1}\) is compact and \(\Phi_m\) is continuous, it follows from (ii) that
\[
        \Phi_m(y)\le  -\min_{x\in S^{d-1}}(-\Phi_m(x))=:-\tilde{c}_1<0
        \qquad \text{for every } y\in S^{d-1}.
\]
Thus by homogeneity of \(\Phi_m\),
\[
        \Phi_m(x)=\|x\|^m\Phi_m(\frac{x}{\|x\|}) \le  -\tilde{c}_1\, \|x\|^m\qquad \text{for every } x\not= 0. 
\]
Since $\sum_{k=0}^{m-1}\Phi_k$ is at most of degree $m-1$, it holds $|\sum_{k=0}^{m-1}\Phi_k(x)|\le \tilde{c}_2(1+\|x\|^{m-1})$, which can be absorbed by $\Phi_m$. Thus by continuity of $\Phi_m$ in zero, it follows for some $\tilde{c}_3, r>0$
\[
        \Phi(x)\le -\tilde{c}_3\, \|x\|^m
        \qquad \text{whenever }\ \|x\|\ge r.
\]
This clearly implies \eqref{prop:eq0}.\\
By Remark \ref{rem:polydifdrift}(i),
\[
L^{A,\mathbf{P}}V = -\frac12\mathrm{trace}(A\nabla^2\Phi)-\langle  A\nabla\Phi, \nabla\Phi\rangle+\frac12\mathrm{div}(\mathbf{B}).
\]
Consider the leading homogeneous part $\Phi_m$ of $\Phi$. By $(ii)$ and Euler's homogeneous function theorem, we have for any $x\in S^{d-1}$ that $\langle x, \nabla\Phi_m(x)\rangle =m\Phi_m(x)<0$. Thus $\nabla\Phi_m(x)\neq 0$ for all $x\in S^{d-1}$. By compactness of $S^{d-1}$, continuity of $\nabla \Phi_m$, and positive definiteness of $A$, there is $c_0>0$ such that
\[
\langle A\nabla\Phi_m(x), \nabla\Phi_m(x)\rangle\ge c_0\qquad \forall x\in S^{d-1}.
\]
Since $\nabla\Phi_m$ is homogeneous of order $m-1$ and $x\mapsto \langle A\nabla\Phi_m(x), \nabla\Phi_m(x)\rangle$ is continuous in zero, we get
\[
\langle A\nabla\Phi_m(x), \nabla\Phi_m(x)\rangle\ge c_0\|x\|^{2m-2}\qquad \forall x\in \R^{d},
\]
and $|\nabla\Phi_m(x)|\le \tilde{c}_1\|x\|^{m-1}$.  Since $\nabla\Phi-\nabla\Phi_m$ has degree at most $m-2$, we get $|(\nabla\Phi-\nabla\Phi_m)(x)|\le \tilde{c}_2(1+\|x\|^{m-2})$. Hence for suitable $c_1,c_2>0$ and all $x\in \R^d$
\begin{eqnarray*}
\langle A\nabla\Phi(x), \nabla\Phi(x)\rangle&\ge &c_0\|x\|^{2m-2}+ 2\langle A(\nabla\Phi-\nabla\Phi_m)(x), \nabla\Phi_m(x)\rangle \\
&\ge&  c_0\|x\|^{2m-2}- 2\|A\| \tilde{c}_2(1+\|x\|^{m-2})\tilde{c}_1\|x\|^{m-1}\\
&\ge &  c_1\|x\|^{2m-2}-c_2.
\end{eqnarray*}
Moreover,
\[
\mathrm{deg}\big (\mathrm{trace}(A\nabla^2\Phi)\big )\le m-2,
\qquad
\mathrm{deg}\big (\mathrm{div}(\mathbf{B})\big )\le \mathrm{deg} (\mathbf{B})-1\le 2m-3,
\]
hence for all $x\in \R^d$
\[
 -\frac12\mathrm{trace}(A\nabla^2\Phi(x))+\frac12\mathrm{div}(\mathbf{B})(x)\le c_3(1+\|x\|^{2m-3}).
\]
Therefore, for $x$ sufficiently large
\[
L^{A,\mathbf{P}}V(x)\le -\frac{c_1}{2}\|x\|^{2m-2}\le V(x)
\]
and \eqref{prop:eq2} holds. Finally, exactly in the same way, we can show that outside some compact set $L^{\widetilde{\mathbf P}}V\le V$ and therefore the stated conservativeness holds. 
\end{proof}

\begin{remark}
Assumption (ii) in Proposition \ref{prop:lyap-criterion} is a sufficient condition on the leading homogeneous part of $\Phi$, not a consequence of $-\Phi(x)\to +\infty$ as $\|x\|\to +\infty$. For instance, $\Phi(x,y)=-(x^2+y^4)$ satisfies $-\Phi(x)\to +\infty$ as $\|x\|\to +\infty$, but $\Phi_4(x,y)=-y^4$ vanishes on the two points $(\pm 1,0)$ of $S^1$. In this situation, the Lyapunov estimate may still work, but one would rather need a direction dependent argument instead of the leading homogeneous part condition Proposition \ref{prop:lyap-criterion}(ii).
\end{remark}

\begin{example}
Let $d=2$ and $\Phi(x,y)=-x^{10}-y^{10}+p(x,y)$, where $p(x,y)$ is a polynomial with $\deg(p)\leq 9$. Thus condition (ii) of Proposition \ref{prop:lyap-criterion} is satisfied.
Let $A=id$, and let $q(x,y)$ be a polynomial with $\deg(q)\le 9$.
For each $(x,y)\in\mathbb{R}^d$, define
$$
C(x,y)=\begin{pmatrix}0 & q(x,y)\\ -q(x,y)&0 \end{pmatrix},
\qquad
\mathbf{P}(x,y)=\nabla\Phi(x,y)+C(x,y)^{T}\nabla\Phi(x,y)+\frac12\,\mathrm{div}\,C(x,y).
$$
Then, $\mathbf{B}=\mathbf{P}-\nabla \Phi=C(x,y)^{T}\nabla\Phi(x,y)+\frac12\,\mathrm{div}\,C(x,y)$, and hence ${\rm deg}(\mathbf{B})\leq 18$. Thus condition (i) of Proposition \ref{prop:lyap-criterion} is satisfied. Moreover, by an explicit direct calculation (using Proposition \ref{suffnececondOHHD}), \eqref{prop:eq1} holds. Thus, by Proposition \ref{prop:lyap-criterion}, the regularized semigroup $(P_t)_{t>0}$ associated with the $L^{1}$-closed extension of $(L^{A,\mathbf{P}},C_0^{\infty}(\mathbb{R}^d))$, where
$$
L^{\mathbf{P}}f=\frac12\Delta f+\langle \mathbf{P},\nabla f\rangle,
\qquad f\in C_0^{\infty}(\mathbb{R}^d),
$$
is conservative (even recurrent), and $\mu:=e^{2\Phi}\,dx$ is finite and the unique (infinitesimally) invariant measure.
\end{example}

\subsection{Quadratic vector fields in $\R^2$ (and short outlook to arbitrary orders)}
In this subsection, we will study explicitly a special case of Section \ref{sec:3.1} with $m=d=2$, and mainly $n=2$. Consider a polynomial vector field $\mathbf{P}$ of order (or also called degree) $n\ge 0$ in $\R^2$, i.e.
\[
\mathbf{P}=\sum_{k=0}^n\mathbf{P}_k
\]
where each $\mathbf{P}_k=(P_{k,1},P_{k,2})$ is a pair of homogeneous polynomials of degree $k$ or zero,
\[
P_{k,1}(z)=\sum_{j=0}^{k}\alpha_{k,j}x^{k-j}y^{j},\qquad P_{k,2}(z)=\sum_{j=0}^{k}\beta_{k,j}x^{k-j}y^{j}, \quad z=(x,y)\in \R^2, \ 0\le k\le n,
\]
with $\alpha_{k,j}, \beta_{k,j}\in \R$, and $P_{n,1}\not \equiv 0$ or $P_{n,2}\not \equiv 0$. Since $\mathbf{P}_1$ is linear, we will use our previous notation for linear vector fields, i.e. for arbitrary $a,b,c,d\in \R$
\[
\mathbf{P}_1(z)=Gz=\begin{pmatrix}a&b\\ c&d\end{pmatrix}\begin{pmatrix}x\\ y \end{pmatrix}=(ax+by,cx+dy), \qquad z=(x,y)\in \R^2.
\] 
Let $A$ be an arbitrary symmetric and positive definite $2\times 2$ matrix. According to Proposition \ref{linearSOHHD}, we can choose a (subsequently fixed) symmetric matrix 
\[
S=\begin{pmatrix}s_{11}& s_{12}\\ s_{12}& s_{22}\end{pmatrix},
\]
such that
$$
\tilde Gz=Sz+(G-AS)z, \quad z=(x,y)\in \R^2
$$ 
is a SOHHD and 
\begin{eqnarray}\label{2diminv}
\mu=e^{2\Phi(x,y)}dxdy,\qquad \Phi(x,y)=\frac{s_{11}x^2+2s_{12}xy+s_{22}y^2}{2},
\end{eqnarray}
is an infinitesimally invariant measure for $\big (\frac{1}{2} \mathrm{trace}(A\nabla^2)+ \langle \mathbf{P}_1, \nabla \cdot \rangle, C_0^{\infty}(\R^2)\big )$. Our goal is to find conditions under which $\mu$ is also an infinitesimally invariant measure for $\big (\frac{1}{2} \mathrm{trace}(A\nabla^2)+ \langle \mathbf{P}, \nabla \cdot \rangle, C_0^{\infty}(\R^2)\big )$ for arbitrary $n\in \N$. We first consider $n=2$. Setting
\begin{eqnarray*}\label{defHdim2}
\mathbf{H}:=\mathbf{P}_0+\mathbf{P}_2= (\alpha_{0,0}+\sum_{j=0}^{2}\alpha_{2,j}x^{2-j}y^{j},\beta_{0,0}+\sum_{j=0}^{2}\beta_{2,j}x^{2-j}y^{j})
\end{eqnarray*}
we have for $z=(x,y)\in \R^2$, 
\begin{eqnarray*}\label{linplusquadplusconst}
\mathbf{P}(z)
& =& G z+\mathbf{H}(z)\nonumber \\
&=& (ax+by,cx+dy)+(\alpha_{0,0}+\alpha_{2,0}x^{2}+\alpha_{2,1}xy+\alpha_{2,2}y^{2},\beta_{0,0}+\beta_{2,0}x^{2}+\beta_{2,1}xy+\beta_{2,2}y^{2}).\nonumber 
\end{eqnarray*}
For $\mu$ to be an infinitesimally invariant measure for 
$$
\Big (\frac{1}{2} \mathrm{trace}(A\nabla^2) + \langle \mathbf{P}, \nabla \cdot \rangle, C_0^{\infty}(\R^2)\Big ),
$$
we need to check under which conditions on $\alpha_{0,0}, \alpha_{2,0}, \alpha_{2,1}, \alpha_{2,2}, \beta_{0,0}, \beta_{2,0}, \beta_{2,1}, \beta_{2,2}$, we have that $\mathrm{div}_{\mu}(\mathbf{H})=0$, where $\mu$ is as in \eqref{2diminv} (cf. Theorem \ref{propstrohelm}(i)). By Theorem \ref{propstrohelm}(ii) the latter is equivalent to
\begin{eqnarray}\label{invariance P_2}
\langle \nabla \Phi, \mathbf{H} \rangle +\frac 12\mathrm{div} \mathbf{H}= 0,
\end{eqnarray}
which further holds, if and only if
\begin{eqnarray*}\label{linplusquadplusconst1}
&&(s_{11}x+s_{12} y, s_{12} x +s_{22}y)(\alpha_{0,0}+\alpha_{2,0}x^{2}+\alpha_{2,1}xy+\alpha_{2,2}y^{2},\beta_{0,0}+\beta_{2,0}x^{2}+\beta_{2,1}xy+\beta_{2,2}y^{2})\nonumber \\
&& \quad + \frac12(2\alpha_{2,0}x+\alpha_{2,1}y+2\beta_{2,2} y+\beta_{2,1} x)\nonumber \\
&&= (s_{11}\alpha_{2,0}+s_{12}\beta_{2,0})x^3+(s_{12}\alpha_{2,2}+s_{22}\beta_{2,2}) y^3+(s_{11} \alpha_{2,2} +s_{12}\alpha_{2,1}+s_{12}\beta_{2,2}+s_{22}\beta_{2,1} )xy^2\nonumber\\
&&\quad + (s_{11} \alpha_{2,1}+s_{12}\alpha_{2,0}+s_{12}\beta_{2,1}+s_{22}\beta_{2,0})x^2y +(s_{11}\alpha_{0,0}+s_{12}\beta_{0,0}+\alpha_{2,0}+\frac{\beta_{2,1}}{2})x+\nonumber\\
&& \quad  +(s_{12}\alpha_{0,0}+s_{22}\beta_{0,0}+\beta_{2,2} +\frac{\alpha_{2,1}}{2})y.
\end{eqnarray*}
Therefore, we obtain the following system of equations
\begin{equation}\label{linplusquadplusconst2}
\begin{cases}
\ s_{11}\alpha_{2,0}+s_{12}\beta_{2,0}=0, \\[6pt]
\ s_{12}\alpha_{2,2}+s_{22}\beta_{2,2}=0, \\[6pt]
\  s_{11}\alpha_{2,2}+s_{12}\alpha_{2,1}+s_{12}\beta_{2,2}+s_{22}\beta_{2,1} =0, \\[6pt]
\ s_{11} \alpha_{2,1}+s_{12}\alpha_{2,0}+s_{12}\beta_{2,1}+s_{22}\beta_{2,0} =0, \\[6pt]
\  s_{11}\alpha_{0,0}+s_{12}\beta_{0,0}+\alpha_{2,0}+\frac{\beta_{2,1}}{2}=0, \\[6pt]
\  s_{12}\alpha_{0,0}+s_{22}\beta_{0,0}+\beta_{2,2} +\frac{\alpha_{2,1}}{2}=0.
\end{cases}
\end{equation}
Note that
\begin{eqnarray}\label{div zero d=2}
\frac 12\mathrm{div} \mathbf{H}(x)= (\alpha_{2,0}+\frac{\beta_{2,1}}{2})x+(\beta_{2,2} +\frac{\alpha_{2,1}}{2})y= 0\ \Longleftrightarrow \ \alpha_{2,0}=-\frac{\beta_{2,1}}{2}\ \text{ and } \ \beta_{2,2} =-\frac{\alpha_{2,1}}{2}.
\end{eqnarray}
In particular, by Theorem \ref{propstrohelm}(ii) the decomposition $\nabla \Phi+\mathbf{H}$ is a SOHHD, if and only if \eqref{linplusquadplusconst2} and \eqref{div zero d=2} hold. 
Thus, we obtain:
\begin{lemma}\label{allinvariantmeasures0}
Let $A, G, S$, $\mu$, be as in Proposition \ref{linearSOHHD}, 
and for $z=(x,y)\in \R^2$
$$
\mathbf{P}(z)=\sum_{k=0}^2\mathbf{P}_k(z)= Gz+\mathbf{H}(z)= (ax+by,cx+dy)+\mathbf{H}(x,y),
$$ 
where
$$
\mathbf{H}(x,y)=(\alpha_{0,0}+\alpha_{2,0}x^{2}+\alpha_{2,1}xy+\alpha_{2,2}y^{2},\beta_{0,0}+\beta_{2,0}x^{2}+\beta_{2,1}xy+\beta_{2,2}y^{2}).
$$
Then, the following holds:
(i) The real scalars $s_{11}, s_{12}, s_{22}, \alpha_{0,0}, \alpha_{2,0}, \alpha_{2,1}, \alpha_{2,2}, \beta_{0,0}, \beta_{2,0}, \beta_{2,1}, \beta_{2,2}$, satisfy \eqref{linplusquadplusconst2}, if and only if  $\mathrm{div}_{\mu}(\mathbf{H})=0$, if and only if \eqref{invariance P_2} holds, if and only if $\mu$ is an infinitesimally invariant measure for 
$$
\big (L^{A,\mathbf{P}}, C_0^{\infty}(\R^2)\big )=\Big (\frac{1}{2} \mathrm{trace}(A\nabla^2) + \big \langle \mathbf{P}, \nabla \cdot \big \rangle, C_0^{\infty}(\R^2)\Big ).
$$ 
Moreover, in any of these cases the vector field decomposition
\begin{eqnarray}\label{OHHD2dim}
\tilde{\mathbf{P}}(z) :=Sz+\mathbf{P}(z)-ASz=\nabla \Phi(x,y)+(G-AS)\begin{pmatrix}x\\ y \end{pmatrix}+\mathbf{H}(x,y), 
\end{eqnarray}
is a SOHHD, if and only if $\alpha_{2,0}=-\frac{\beta_{2,1}}{2}$ and $\beta_{2,2} =-\frac{\alpha_{2,1}}{2}$, and
\[
\tilde Gz=Sz+(G-AS)z, \quad z=(x,y)\in \R^2
\] 
is a SOHHD by assumption.\\
(ii) Suppose that $\mu$ is an infinitesimally invariant measure for $(L^{A,\mathbf{P}}, C_0^{\infty}(\R^2))$. Then, the measure $\mu$ is infinitesimally invariant for $(L, C_0^{\infty}(\R^2))$, where $L$ is either of the following operators
\[
L^{A,G}, L^{S}, L^{\widetilde{G}}, L^{\widetilde{\mathbf{P}}}.
\]
\end{lemma}
\begin{proof}
(i) has been derived in the paragraphs preceding the lemma. By assumption and (i), we have that $\mathrm{div}_{\mu}(\mathbf{H})=0$. The statement is true for $L=L^{A,G}$ by construction (Proposition \ref{linearSOHHD}). Obviously $L^S$ is symmetric with respect to $\mu$, so it also holds for $L=L^S$, which further implies it is true for $L=L^{\widetilde{G}}$, since $\tilde Gz=Sz+(G-AS)z$ is a SOHHD, hence $\mathrm{div}_{\mu}(G-AS)=0$. Finally it is also true for $L=L^{\widetilde{\mathbf{P}}}$ since it is true for $L=L^{\widetilde{G}}$ and $\mathrm{div}_{\mu}(\mathbf{H})=0$, as well as $\widetilde{\mathbf{P}}=\widetilde{G}+\mathbf{H}$.
\end{proof}
\begin{proposition}\label{allinvariantmeasures}
Consider the situation of Lemma \ref{allinvariantmeasures0}. The following are equivalent:
\begin{itemize}
\item[(a)] $G=\mathbf{P}_1$ is Hurwitz (so that $S$ is negative definite), and $\mu$ is an infinitesimally invariant measure for $\big (L^{A,\mathbf{P}}, C_0^{\infty}(\R^2)\big )$. 
\item[(b)] $\mu$ is a finite invariant measure for the semigroup associated to the closed extension of $(L^{A,\mathbf{P}},C_0^{\infty}(\mathbb{R}^2))$ and its associated stochastic process, both constructed in \cite{LST22}. 
\end{itemize}
Moreover, in either case (a) or (b), the following hold: 
\begin{itemize}
\item[(i)]
$\mu$ is the (up to a multiplicative constant) the unique finite invariant measure for the semigroup mentioned in (b), and the unique finite infinitesimally invariant measure for  $(L^{A,\mathbf{P}},C_0^{\infty}(\mathbb{R}^2))$.
\item[(ii)]
If $\alpha_{2,0}\not =-\frac{\beta_{2,1}}{2}$ or $\beta_{2,2} \not=-\frac{\alpha_{2,1}}{2}$, then there is no SOHHD for $\tilde{\mathbf{P}}$.
\item[(iii)]
If $\alpha_{2,0}=-\frac{\beta_{2,1}}{2}$ and $\beta_{2,2} =-\frac{\alpha_{2,1}}{2}$ then
\eqref{OHHD2dim} is the unique OHHD for $\tilde{\mathbf{P}}$.
\end{itemize}
\end{proposition}
\begin{proof}
(a) $\Rightarrow$ (b): Since $\mu$ is an infinitesimally invariant measure for $(L^{A,\mathbf{P}}, C_0^{\infty}(\R^2))$, it follows from Lemma \ref{allinvariantmeasures0}(ii) and Theorem \ref{propstrohelm}(i), that \eqref{prop:eq1} holds. Moreover, since $G$ is Hurwitz, it follows by Proposition \ref{propintegrep}(ii) that $S$ is negative definite, so that $\mu=\exp(\langle Sx,x\rangle)dx$ is finite and $\Phi(x)=\frac12\langle Sx,x\rangle$ has only strictly negative values on the sphere $S^1$. Obviously, for  $m:=\mathrm{deg}(\Phi)=2$, we have $q:=\mathrm{deg}(\mathbf{P}-A\nabla \Phi)\le 2m-2=2$. Therefore the assertion follows from Proposition \ref{prop:lyap-criterion}.\\
(b) $\Rightarrow$ (a): Since $\mu$ is finite, we must have that $S$ is negative definite and moreover by \cite[Lemma 3.16]{LT22} $\mu$ is an infinitesimally invariant measure for $(L^{A,\mathbf{P}}, C_0^{\infty}(\R^2))$. By construction, $S$ solves \eqref{SOHHDlinalg1}, which we can multiply from the left and from the right with $S^{-1}$ to see that $P=-S^{-1}$ is a solution to the Lyapunov equation
\[
PG^T+GP=-2A,
\]
and therefore $G^T$ is Hurwitz by \cite[Theorem 4.6]{K02}. But then also $G$ is Hurwitz.\\
Using (a), (b) and \cite[Proposition 3.13, Theorems 3.15 and 3.17]{LT22} we obtain both, the uniqueness of $\mu$ as invariant and  as infinitesimally invariant measure. For the last statement on the OHHD, we first note that $\mu$ is an infinitesimally invariant measure of $(L^{\widetilde{\mathbf{P}}},C_0^{\infty}(\mathbb{R}^2))$ by Lemma \ref{allinvariantmeasures0}(ii). Moreover $\mu$ is finite and similarly to the proof (a) $\Rightarrow$ (b), it can be shown that conservativeness holds for $(L^{\widetilde{\mathbf{P}}},C_0^{\infty}(\mathbb{R}^2))$, thus $\mu$ is the unique infinitesimally invariant measure for $(L^{\widetilde{\mathbf{P}}},C_0^{\infty}(\mathbb{R}^2))$, hence  Theorem \ref{nonexofOHHD} applies, noting that $\mathrm{div}(\widetilde{\mathbf{P}}-\nabla \Phi)=0$ if and only if $\mathrm{div}(\mathbf{H})=0$.
\end{proof}
\begin{example}\label{exam:quadnonOHHD}
Let
\[
G=\begin{pmatrix}0&1\\ -1&-3\end{pmatrix},\quad \text{and}\quad A=\begin{pmatrix}1&-1\\ -1&2\end{pmatrix}.
\] 
Then $G$ is a Hurwitz matrix, $A$ is symmetric and positive definite and one can readily check that $G=AS$, with
\[
S=\begin{pmatrix}-1&-1\\ -1&-2\end{pmatrix},
\]
being negative definite. Then obviously $\mathrm{trace}(G-AS)=0$ and $\langle Sz,(G-AS)z\rangle=0$ and it follows by Proposition \ref{propintegrep} that
\begin{eqnarray*}
\mu=e^{2\Phi(x,y)}dxdy=e^{-x^2-2xy-2y^2}dxdy,
\end{eqnarray*}
is the unique infinitesimally invariant measure for $(L^{A,G},C_0^{\infty}(\mathbb{R}^d))$ and the unique (finite) invariant measure for the semigroup associated to the closed extension of $(L^{A,G},C_0^{\infty}(\mathbb{R}^d))$ and its associated stochastic process, constructed in \cite{LST22}. In particular, 
$$
\tilde Gz=Sz+(G-AS)z=Sz=\nabla \Phi(z)
$$
is the unique SOHHD of $\tilde Gz$. Now we consider an antisymmetric  polynomial perturbation $\mathbf{H}$ of the $\mu$-symmetric operator $L^{A,G}$ without changing the unique invariant measure $\mu$.
Let for $z=(x,y)\in \R^2$
\begin{eqnarray*}
\mathbf{H}(z)\ =\ \mathbf{P}_0(z)+\mathbf{P}_2(z)&=&(\alpha_{0,0}+\alpha_{2,0}x^{2}+\alpha_{2,1}xy+\alpha_{2,2}y^{2},\beta_{0,0}+\beta_{2,0}x^{2}+\beta_{2,1}xy+\beta_{2,2}y^{2})\\
&=&(3/2+2x^{2}+xy-6y^{2},1-2x^{2}+xy+3y^{2}).
\end{eqnarray*}
Then one can easily check that \eqref{invariance P_2}, hence \eqref{linplusquadplusconst2} holds, and one has
\[
\mathbf{P}(z)=Gz+\mathbf{H}(z)=(3/2+y+2x^{2}+xy-6y^{2},1-x-3y-2x^{2}+xy+3y^{2}).
\]
Thus by Proposition \ref{allinvariantmeasures}, $\mu$ is the unique infinitesimally invariant measure for $\big (L^{A,\mathbf{P}}, C_0^{\infty}(\R^2)\big )$ and the unique (finite) invariant measure for the semigroup associated to the closed extension of $(L^{A,\mathbf{P}},C_0^{\infty}(\mathbb{R}^d))$ and its associated stochastic process, constructed in \cite{LST22}. Since moreover
\[
\frac 12\mathrm{div} ( \mathbf{H}(z))=\frac12(4x+y+x+6y)=\frac52 x+\frac72 y\not =0,
\]
i.e \eqref{div zero d=2} is not satisfied, it follows from Proposition \ref{allinvariantmeasures} that 
 \begin{eqnarray*}
\tilde{\mathbf{P}}(z) :=Sz+\sum_{k=0}^2\mathbf{P}_k (z)-ASz=\nabla \Phi(z)+\mathbf{H}(z), 
\end{eqnarray*}
 is not an SOHHD and that there does not exist an SOHHD for $\tilde{\mathbf{P}}$. 
\end{example} 
\medskip     
For $n=2$, we saw $\mu$ is infinitesimally invariant for $\big (\frac{1}{2}\mathrm{trace}(A\nabla^2)+ \langle \sum_{k=0}^n\mathbf{P}_k, \nabla \cdot \rangle, C_0^{\infty}(\R^2)\big )$, if and only if 
\begin{enumerate}
\item[1.] $\langle Sz, \mathbf{P}_0(z) \rangle +\frac 12\mathrm{div} ( \mathbf{P}_2(z))= 0$ (homogeneous of degree $1$),
\item[2.] $\frac 12\mathrm{div} ( \mathbf{P}_2(z))=0$ (homogeneous of degree $2$).
\end{enumerate}
For $n\ge 3$, $\mu$ is infinitesimally invariant for $\big (\frac{1}{2}\mathrm{trace}(A\nabla^2)+ \langle \sum_{k=0}^n\mathbf{P}_k, \nabla \cdot \rangle, C_0^{\infty}(\R^2)\big )$, if and only if 
$\mathrm{div}_{\mu}(\mathbf{P}_0+\mathbf{P}_2+\ldots + \mathbf{P}_n)=0$. \\[3pt]
This gives for $n=3$
\begin{enumerate}
\item[1.] $\langle Sz, \mathbf{P}_0(z)\rangle +\frac 12\mathrm{div} ( \mathbf{P}_2(z))= 0$ (homogeneous of degree $1$),
\item[2.] \hspace*{+2cm}$\frac 12\mathrm{div} ( \mathbf{P}_3(z))=0$ (homogeneous of degree $2$),
\item[3.] $\langle Sz, \mathbf{P}_{2}(z) \rangle=0$ \hspace*{+2.15cm}(homogeneous of degree $3$),
\item[5.] $\langle Sz, \mathbf{P}_{3}(z) \rangle=0$ \hspace*{+2.15cm}(homogeneous of degree $4$),
\end{enumerate}
and for general $n\ge 4$
\begin{enumerate}
\item[1.] $\langle Sz, \mathbf{P}_0(z) \rangle +\frac 12\mathrm{div} ( \mathbf{P}_2(z))= 0$ (homogeneous of degree $1$),
\item[2.] \hspace*{+2cm}$\frac 12\mathrm{div} ( \mathbf{P}_3(z))=0$ (homogeneous of degree $2$),
\item[3.] $\langle Sz, \mathbf{P}_2(z) \rangle+\frac 12\mathrm{div} ( \mathbf{P}_4(z))=0$ (homogeneous of degree $3$),\\
\qquad\qquad\qquad\vdots\qquad\qquad\qquad\vdots
\item[4.] $\langle Sz, \mathbf{P}_{n-2}(z) \rangle+\frac 12\mathrm{div}(\mathbf{P}_{n}(z) )= 0$ (homogeneous of degree $n-1$),
\item[5.] $\langle Sz, \mathbf{P}_{n-1}(z) \rangle=0$ \hspace*{+2.2cm}(homogeneous of degree $n$),
\item[6.] $\langle Sz, \mathbf{P}_{n}(z) \rangle=0$ \hspace*{+2.6cm}(homogeneous of degree $n+1$).
\end{enumerate}

\subsection{Nonlinear vector fields via anti-symmetric matrices of functions}

\begin{lemma} \label{conservlema}
Let $\mathbf{G} \in L^p_{loc}(\mathbb{R}^d, \mathbb{R}^d)$ for some $p\in(d,\infty)$, $A=(a_{ij})_{1\le i,j\le d}$ be a symmetric and positive definite matrix of real numbers,  and let $(P_t)_{t>0}$ be the regularized semigroup associated with the $L^1$-closed extension of $(L^{A,\mathbf{G}}, C_0^{\infty}(\mathbb{R}^d))$, defined by
$$
L^{A,\mathbf{G}}f = \frac12 \mathrm{trace}(A\nabla^2 f) + \langle \mathbf{G}, \nabla f \rangle, \quad f \in C_0^{\infty}(\mathbb{R}^d)
$$
(see \cite{LST22}). Let $T$ be a $d \times d$ strictly positive definite symmetric matrix. If there exist constants $N_0\in \N$ and $M>0$ such that
$$
\langle \mathbf{G}(x), Tx \rangle \leq M \|x\|^2, \quad \text{for a.e. } x \in \mathbb{R}^d \setminus \overline{B}_{N_0},
$$
then $(P_t)_{t>0}$ is conservative.
\end{lemma}
\begin{proof}
Consider the Lyapunov function defined by
$$
V(x) = 1 + \langle Tx, x \rangle, \quad x \in \mathbb{R}^d.
$$
Since $T$ is strictly positive definite, it holds that $V(x) \geq 1$ for all $x \in \mathbb{R}^d$, and
$$
\lim_{\|x\| \to \infty} V(x) = \infty.
$$
A direct calculation yields
$$
\nabla V(x) = 2Tx, \quad \nabla^2 V(x) = 2T,
$$
and we obtain
\begin{align*}
L^{A,\mathbf{G}}V(x) &= \frac12 \mathrm{trace}(A\nabla^2 V(x)) + \langle \mathbf{G}(x), \nabla V(x) \rangle \\
&= \mathrm{trace}(AT) + 2\langle \mathbf{G}(x), Tx \rangle.
\end{align*}
Let $m$ be the minimum eigenvalue of $T$. By the assumption on $\mathbf{G}$, we get for a.e. $x \in \mathbb{R}^d \setminus \overline{B}_{N_0}$,
$$
L^{A,\mathbf{G}}V(x) \leq  \mathrm{trace}(AT) + 2M \|x\|^2 \leq \mathrm{trace}(AT) + \frac{2M}{m} \langle Tx, x \rangle \leq \hat{M} V(x), 
$$
where $\hat{M} := \max\big (\mathrm{trace}(AT), \frac{2M}{m}\big )$. Thus, \cite[Lemma 3.26, Corollary 3.23]{LST22} imply that the associated semigroup $(P_t)_{t>0}$ is conservative.
\end{proof}\\
In the following Proposition \ref{suffnececondOHHD}, similarly to  \cite{HHS05}, we present a class of general vector fields that are OHHDs, hence WHHDs.
\begin{proposition}\label{suffnececondOHHD}
Let $\Phi \in H^{1,2}_{loc}(\mathbb{R}^d)\cap L^{\infty}_{loc}(\mathbb{R}^d)$, $c_{ij}=-c_{ji}\in H^{1,2}_{loc}(\R^d)\cap L^{\infty}_{loc}(\mathbb{R}^d),\ 1\le i,j\le d$, and $\mu=\rho\,dx$, where $\rho=\exp(2 \Phi)$. Set 
$$
\overline{\mathbf{B}}:= C^T\nabla \Phi.
$$ 
(i) We have the following equivalences:
$$
\mathrm{div}_{\mu}(\mathrm{div}C)= 0 \ \Leftrightarrow \  \mathrm{div}_{\mu}(\overline{\mathbf{B}})= 0 \ \Leftrightarrow\ \big \langle \mathrm{div} C, \nabla \Phi\big \rangle=0\ \Leftrightarrow\ \mathrm{div}_{dx}(\overline{\mathbf{B}})= 0.
$$ 
(ii) Let $k_1, k_2 \in \mathbb{R}$ be constants. Then, we have:
\begin{enumerate}
\item [(a)] If either one of the four equivalent conditions in (i) holds, then 
\[
\mathbf{G} =  \nabla \Phi+ k_1C^T\nabla \Phi+ k_2\mathrm{div} C
\]
is an OHHD.
\item [(b)] If either $k_1\neq 0$ or $k_2 \neq 0$, and $\mathbf{G} =  \nabla \Phi+ k_1C^T\nabla \Phi+ k_2\mathrm{div} C$ is an OHHD, then one of the four equivalent conditions of (i) is satisfied.
\end{enumerate}
\end{proposition}
\begin{proof}
(i) Since $\beta^{\rho, C^T}=\frac 12\mathrm{div} C+ \overline{\mathbf{B}}$ is $\mu$-divergence free (cf. e.g. \cite[Remark 2.28]{LST22} or just calculate it), $\overline{\mathbf{B}}$ is $\mu$-divergence free, if and only if $\mathrm{div} C$ is $\mu$-divergence free. This shows the first equivalence. The remaining two equivalences are shown as follows: Let $(\Phi_m)_{m\geq 1}$ be a mollification of $\Phi$, and for each $1 \leq i,j \leq d$, let $(c_{ij,n})_{n\geq 1}$ be a mollification of $c_{ij}$. Then for any $f \in C_0^{\infty}(\R^d)$
\begin{eqnarray*}
\int_{\R^d}\langle \overline{\mathbf{B}}, \nabla f\rangle \rho \, dx&=&
\lim_{m \rightarrow \infty} 
\lim_{n \rightarrow \infty} \int_{\R^d} e^{2\Phi_m}\big \langle (\sum_{j=1}^d c_{1j,n}^T\partial_j \Phi_m,\ldots,\sum_{j=1}^d c_{dj, n}^T\partial_j \Phi_m), \nabla f\big \rangle\, dx \\
&=&
\lim_{m \rightarrow \infty} 
\lim_{n \rightarrow \infty} -\int_{\R^d}\Big (  \sum_{i=1}^d  \partial_i \big ( e^{2\Phi_m} \sum_{j=1}^d c_{ij}^T\partial_j \Phi_m  \big )\Big )f\, dx \\
&=&\lim_{m \rightarrow \infty} 
\lim_{n \rightarrow \infty}-\int_{\R^d} e^{2\Phi_m}\Big (\sum_{i,j=1}^d \partial_i c_{ij, n}^T\partial_j \Phi_m \Big )f\, dx\\
&=&\int_{\R^d} e^{2\Phi}\Big (\sum_{i,j=1}^d \partial_i c_{ji}^T\partial_j \Phi \Big )f\, dx\ =\ \int_{\R^d}e^{2\Phi}\big \langle \mathrm{div} C, \nabla \Phi\big \rangle f \, dx,
\end{eqnarray*}
and similarly, for any $f \in C_0^{\infty}(\R^d)$
\begin{align*}
\int_{\R^d}\langle \overline{\mathbf{B}}, \nabla f\rangle \, dx=\int_{\R^d} \big \langle \mathrm{div} C, \nabla \Phi\big \rangle f \, dx.
\end{align*}
(ii) Obviously $\mathrm{div}_{dx}(\mathrm{div} C)=0$ and $\big \langle \overline{\mathbf{B}}, \nabla \Phi\big \rangle=0$ always hold by the anti-symmetry of $C$. By definition $\nabla \Phi+ k_1\overline{\mathbf{B}}+ k_2\mathrm{div} C$ is an OHHD, if and only if 
$$
\mathrm{div}_{dx}(k_1\overline{\mathbf{B}}+ k_2\mathrm{div} C)= k_1\mathrm{div}_{dx}(\overline{\mathbf{B}})=0\ \text{ and }\ \big \langle k_1\overline{\mathbf{B}}+ k_2\mathrm{div} C, \nabla \Phi\big \rangle=k_2\big \langle \mathrm{div} C, \nabla \Phi\big \rangle=0.
$$
The latter is obviously true, if one of the four equivalent conditions in (i) is satisfied. Conversely, if  either $k_1\neq 0$ or $k_2 \neq 0$, and $\nabla \Phi+ k_1C^T\nabla \Phi+ k_2\mathrm{div} C$ is an OHHD, then either $\mathrm{div}_{dx}(\overline{\mathbf{B}})=0$ or $\big \langle \mathrm{div} C, \nabla \Phi\big \rangle=0$ and therefore 
one of the four equivalent conditions of (i) is satisfied.
\end{proof}

\begin{theorem} \label{mainthm}
Let $C=(c_{ij})_{1 \leq i,j \leq d}$ be a matrix of functions such that $c_{ij}=-c_{ji} \in H^{1,p}_{loc}(\mathbb{R}^d)$ with $p \in (d, \infty)$ for all $1 \leq i,j \leq d$. Let  $G$ be a $d \times d$ matrix of real numbers, $S$ be a symmetric matrix satisfying \eqref{SOHHDlinalg1} with $A=id$ (cf. Proposition \ref{propintegrep}(i)), and $\Phi(x) = \frac12 \langle Sx, x \rangle$. Define the vector field $\mathbf{G}_0$ by
$$
\mathbf{G}_0(x) := Gx + C(x)^T Sx + \frac12 {\rm div} C(x), \quad x \in \mathbb{R}^d.
$$
Then, the following statements hold:
\begin{itemize}
\item[(i)] The decomposition
\begin{equation} \label{decomposipol}
\mathbf{G}_0(x) = \nabla \Phi(x) + \mathbf{B}(x),
\end{equation}
where $\mathbf{B}(x) = (G-S)x + C^T Sx + \frac12 {\rm div} C(x)$, is a WHHD for $\mathbf{G}_0$.
\item[(ii)] Assume that
\begin{equation} \label{orthogonalit}
\langle\mathrm{div} C(x), Sx \rangle = 0 \quad \text{for a.e. } x \in \mathbb{R}^d.
\end{equation}
Then, \eqref{decomposipol} is an OHHD for $\mathbf{G}_0$.
Moreover, for any triple of constants $k_0, k_1, k_2 \in \mathbb{R}$, the vector field defined by
\begin{equation} \label{orthogong1}
\mathbf{G}_1(x):= \nabla \Phi(x)+k_0(G-S)x + k_1C(x)^T Sx +k_2 {\rm div} C(x)
\end{equation}
is an OHHD. 

\item[(iii)] Suppose that \eqref{orthogonalit} holds and that $G$ is a Hurwitz matrix.
Then, for any triple of constants $k_0, k_1, k_2 \in \mathbb{R}$, \eqref{orthogong1} is the unique OHHD for $\mathbf{G}_1$. Furthermore, $\mu(dx) = e^{\langle Sx, x \rangle}\,dx$ is the unique infinitesimally invariant measure for $(L^{\mathbf{G}_1},C_0^{\infty}(\mathbb{R}^d))$ and the unique invariant measure for the semigroup associated to the closed extension of $(L^{\mathbf{G}_1},C_0^{\infty}(\mathbb{R}^d))$ and its associated stochastic process, both constructed in \cite{LST22}. 
\end{itemize}
\end{theorem}
\begin{proof}
First recall that for $\mu(dx)= \rho(x) \,dx$, with $\rho(x)= e^{\langle Sx, x \rangle}$, we have $Sx=\nabla \Phi = \frac{1}{2\rho } \nabla \rho$. \\
(i) To prove the assertion, it suffices to show that $\mathrm{div}_{\mu}(\mathbf{B})=0$.
By Lemma \ref{lemma algebraic riccati}, $Gx=Sx+(G-S)x$ is an OHHD, hence by Theorem \ref{propstrohelm} a WHHD and therefore $\mathrm{div}_{\mu}((G-S)x)=0$. Since moreover,  $\mathrm{div}_{\mu}(C^T Sx + \frac12 \mathrm{div} C(x))=0$ as in the proof of Proposition \ref{suffnececondOHHD},  the assertion follows.
\noindent (ii) The assertion follows directly from Proposition \ref{suffnececondOHHD} and the fact that $\widetilde{G}x=Sx+k_0(G-S)x$ is an OHHD for any $k_0\in \R$.

\noindent (iii) Since $G$ is a Hurwitz matrix, $S$ is a strictly negative definite matrix, and so $\mu(dx)=e^{\langle Sx, x \rangle}\,dx$ is a finite measure. In particular, $\mu$ is an infinitesimally invariant measure for $(L^{\mathbf{G}_1}, C_0^{\infty}(\mathbb{R}^d))$ by (ii) and Theorem \ref{propstrohelm}. Moreover, since $C$ is skew-symmetric, we observe that
$$
\langle C^T Sx, Sx \rangle = 0, \quad \text{for all } x \in \mathbb{R}^d.
$$
Thus for some constant $M>0$,
$$
\langle \mathbf{G}_1(x), -Sx \rangle  = \langle \nabla \Phi(x), -Sx \rangle = \langle Sx, -Sx \rangle \leq M \|x\|^2.
$$
Therefore, by Lemma \ref{conservlema}, the semigroup associated with $(L^{\mathbf{G}_1},C_0^{\infty}(\mathbb{R}^d))$ is conservative. Consequently, the assertion follows from \cite[Proposition 3.13 and Theorems 3.15, 3.17]{LT22}.
\end{proof}

\begin{example}
\begin{itemize}
\item[(i)] Given $S$ as in Theorem \ref{mainthm}, we provide a whole class of anti-symmetric matrices of functions $C=(c_{ij})_{1 \leq i,j \leq d}$ that satisfy \eqref{orthogonalit} and the regularity assumptions of Theorem \ref{mainthm}. Let $\hat{K}=(\hat{k}_{ij})_{1 \leq i,j \leq d}$ be a matrix of functions satisfying $\hat{k}_{ij}=-\hat{k}_{ji} \in C^1(\mathbb{R})$. Define a matrix of functions $C=(c_{ij})_{1 \leq i,j \leq d}$ by
$$
C(x) :=\hat{K}\big(\langle Sx, x \rangle \big), \quad x \in \mathbb{R}^d.
$$
Let $U(x)= \langle Sx, x \rangle$. Then, 
$$
\big(\mathrm{div}C(x) \big)_{j} = \sum_{i=1}^d \partial_i \big( \hat{k}_{ij}(U(x)) \big) = \sum_{i=1}^d \hat{k}'_{ij}(U(x)) \partial_i U(x), \quad  1\le j\le d.
$$
Consequently,
$$
\langle \mathrm{div} C(x), Sx \rangle= \sum_{j=1}^d \sum_{i=1}^d \hat{k}'_{ij}(U(x)) \partial_i U(x) (Sx)_j = \frac{1}{2} \sum_{i,j=1}^d \hat{k}'_{ij}(U(x)) \partial_i U(x) \partial_j U(x)=0.
$$
\item[(ii)] A specific example of a matrix $C$ as in (i) is given as follows: let $K$ be an arbitrary  anti-symmetric $d \times d$ matrix of real numbers and let $\phi: \mathbb{R} \to \mathbb{R}$ be a continuously differentiable function. Then
\begin{eqnarray}\label{exam3.12: eq1}
C(x) := \phi(\langle Sx, x \rangle) K, \quad x \in \mathbb{R}^d
\end{eqnarray}
satisfies the assumptions of (i).\\
For an explicit example of $C$ as in \eqref{exam3.12: eq1} consider the case $d=2$ and assume that $G$ is Hurwitz. Let $S = \begin{pmatrix} s_{11} & s_{12} \\ s_{12} & s_{22} \end{pmatrix}$ be a strictly negative definite matrix solving \eqref{SOHHDlinalg1} with $A=id$. Define
$$
C(x) = \begin{pmatrix} 0 & p(x) \\ -p(x) & 0 \end{pmatrix}, \ \text{where }\ p(x) = \alpha x^2 + \beta xy + \gamma y^2 + \delta x + \varepsilon y.
$$
We want to derive conditions on $C$, more precisely on $p$, so that $\langle \mathrm{div}C(x), Sx \rangle = 0$, i.e. \eqref{orthogonalit} holds.
Now,
$$
\mathrm{div}C(x) = (-\partial_y p, \partial_x p)=( -\beta x - 2\gamma y -\varepsilon, 2\alpha x + \beta y + \delta   )
$$
and
$$
 Sx = ( s_{11}x + s_{12}y, s_{12}x + s_{22}y).
$$
Observe that
\begin{align*}
 (-\beta x - 2\gamma y - \varepsilon )(s_{11}x + s_{12}y) = -\beta s_{11} x^2 - \beta s_{12} xy - 2\gamma  s_{11} xy - 2\gamma  s_{12} y^2 - \varepsilon  s_{11} x - \varepsilon  s_{12} y
\end{align*}
and
\begin{align*}
 (2\alpha x + \beta y + \delta )(s_{12}x + s_{22}y) = 2\alpha s_{12} x^2 + 2\alpha s_{22} xy + \beta s_{12} xy + \beta s_{22} y^2 + \delta  s_{12} x + \delta  s_{22} y.
\end{align*}
By expanding the condition $\langle \mathrm{div}C(x), Sx \rangle = 0$ and collecting terms with respect to the powers of $x$ and $y$, we get
$$
\left\{
\begin{aligned}
 s_{12}\delta   - s_{11}\varepsilon  &= 0, \\
 s_{22}\delta  - s_{12} \varepsilon &= 0.
\end{aligned}
\right.
$$
Since ${\rm det}(S)>0$,
$$
\delta  = 0, \quad \varepsilon = 0.
$$
Next, we compare the coefficients of ($x^2$, $xy$, and $y^2$). Then,
$$
\left\{
\begin{aligned}
  2s_{12}\alpha  - s_{11}\beta  &= 0, \\
  2s_{22}\alpha  -2 s_{11}\gamma  &= 0, \\
  s_{22}\beta  - 2 s_{12}\gamma  &= 0.
\end{aligned}
\right.
$$
Since $S$ is strictly negative definite, we have $s_{11}, s_{22} < 0$. Thus, investigating the second equation, we obtain
$$
\alpha  = k s_{11}, \quad \gamma  = k s_{22}
$$
for some $k \in \mathbb{R}$. Substituting these into the first and third equations, we deduce
$$
\beta  = 2k s_{12}.
$$
Therefore, for some $k \in \mathbb{R}$
$$
p(x) = k(s_{11}x^2 + 2s_{12}xy + s_{22}y^2) = k \langle Sx, x \rangle              
$$            
and we obtain that $C(x)= \phi(\langle Sx, x \rangle) K$, with $\phi(x)=x$ and $K = \begin{pmatrix} 0 & k \\ -k &  0 \end{pmatrix}$.
\item[(iii)]
Let  $G$ be a given $d \times d$ matrix, and $S$ be a symmetric matrix satisfying \eqref{SOHHDlinalg1} with $A=id$ (cf. Proposition \ref{propintegrep}(i)). Based on the framework in (i), we can identify cases where the OHHD structure with potential $\Phi(x) =\frac12 \langle Sx, x \rangle$ is preserved when we add a specific polynomial vector field to the linear vector field $Gx$.\\
For instance, let $d=3$. For arbitrary real numbers $r_{12}, r_{13}, r_{23} \in \mathbb{R}$, let $r_{21}:=-r_{12}$, $r_{31}:=-r_{13}$, $r_{32}:=-r_{23}$, and $r_{11}=r_{22}=r_{33}:=0$. Define the anti-symmetric matrix $\hat{K}(t):=(\hat{k}_{ij})_{1 \leq i,j \leq d}$ as follows:
$$
\hat{k}_{ij}(t)=r_{ij} t^2, \quad t \in \mathbb{R}.
$$
Let $C(x) :=\hat{K}\big(\langle Sx, x \rangle \big)$ for $x \in \mathbb{R}^d$. 
Note that $\hat{k}'_{ij}(t)=2r_{ij} t$ is a linear function of $t \in \mathbb{R}$. Since $U(x)=\langle Sx, x \rangle$ is a quadratic form, $\hat{k}'_{ij}(U(x))$ becomes a homogeneous polynomial of degree 2. Consequently, the term
$$
\big(\mathrm{div} C(x) \big)_{j} = \sum_{i=1}^3 \hat{k}'_{ij}(U(x)) \partial_i U(x) = \sum_{i=1}^3 2r_{ij}\langle Sx, x \rangle \cdot 2(Sx)_i, \quad  1\le j\le d,
$$
is indeed a cubic homogeneous polynomial. 
Therefore, if
$$
P(x):=\operatorname{div} C(x),
$$
then the vector field
$$
Gx+P(x)
=
Sx+\big((G-S)x+P(x)\big)
$$
admits an OHHD with potential $\Phi(x)=\frac12\langle Sx,x\rangle$. Moreover, for any anti-symmetric matrix of functions $\hat{C}=(\hat{c}_{ij})_{1 \leq i,j \leq d}$ satisfying $\mathrm{div}\hat{C}=0$,
$$
Sx+\big( (G-S)x  +P(x) +\hat{C}(x) Sx  \big)
$$
is an OHHD for the vector field $Gx+\mathrm{div}C(x) + \hat{C}(x)Sx$.

\end{itemize}
\end{example}

\section{Algebraic Riccati Equations}\label{section:algRiceq}
\begin{defn} \label{vectorizdefn}
Denote by $M_{d}(\mathbb{R})$ the set of all real $d \times d$ matrices. Let $X \in M_{d}(\mathbb{R})$ with columns $\mathbf{x}_1, \mathbf{x}_2, \dots, \mathbf{x}_d \in \mathbb{R}^d$. The vectorization of $X$, denoted by $\operatorname{Vec}(X)$, is the $d^2 \times 1$ column vector given by
\[
X = \begin{pmatrix}
| & | & & | \\
\mathbf{x}_1 & \mathbf{x}_2 & \dots & \mathbf{x}_d \\
| & | & & |
\end{pmatrix}
\quad \xrightarrow{\quad \text{\rm Vec} \quad} \quad
\text{\rm Vec}(X) = 
\begin{pmatrix}
\mathbf{x}_1 \\ \mathbf{x}_2 \\ \vdots \\ \mathbf{x}_d
\end{pmatrix} \in \mathbb{R}^{d^2}.
\]
Then, the map $\text{Vec}: M_{d}(\mathbb{R}) \to \mathbb{R}^{d^2}$ is indeed a vector space isomorphism.
\end{defn}

\begin{defn} \label{tensorproducde}
Let $A = (a_{ij})_{1 \leq i,j \leq d} \in M_{d}(\mathbb{R})$ and $B \in M_{d}(\mathbb{R})$. The Kronecker product of $A$ and $B$, denoted by $A \otimes B$, is the $d^2 \times d^2$ block matrix defined by:
\[
A \otimes B = \begin{pmatrix}
a_{11} B & a_{12} B & \dots & a_{1d} B \\
a_{21} B & a_{22} B & \dots & a_{2d} B \\
\vdots & \vdots & \ddots & \vdots \\
a_{d1} B & a_{d2} B & \dots & a_{dd} B
\end{pmatrix}.
\]
\end{defn}

\begin{proposition}{\cite[Lemma 4.3.1]{HJ91}} \label{tensform}
For $A, B, X \in M_n(\mathbb{R})$, the following identity holds:
\[ 
\operatorname{Vec}(AXB) = (B^T \otimes A) \operatorname{Vec}(X).
\]
\end{proposition}
\begin{proof}
Let $\mathbf{x}_k$ denote the $k$-th column of the matrix $X$, and let $b_{kj}$ denote the $(k,j)$-component of the matrix $B$. We compare the $j$-th block (of size $m \times 1$) of the vectors on both sides.\\ \\
\noindent \textbf{LHS:} Consider the $j$-th column of the matrix $AXB$. The $j$-th column of $XB$ is given by 
$$
\sum_{k=1}^d  b_{kj} \mathbf{x}_k.
$$
The $j$-th column of $AXB$ is:
\[
A \left( \sum_{k=1}^d  b_{kj} \mathbf{x}_k \right) = \sum_{k=1}^d b_{kj} (A\mathbf{x}_k).
\]

\noindent \textbf{RHS:}  The $(j, k)$-th block of $B^T \otimes A$ is given by $(B^T)_{jk} A = b_{kj} A$. Meanwhile, $\operatorname{Vec}(X)$ consists of $\begin{pmatrix}
\mathbf{x}_1 \\ \mathbf{x}_2 \\ \vdots \\ \mathbf{x}_d
\end{pmatrix}$. The $j$-th block of the matrix-vector product $(B^T \otimes A)\operatorname{Vec}(X)$ is obtained by
\[
\sum_{k=1}^d (b_{kj} A) \mathbf{x}_k = \sum_{k=1}^d b_{kj} (A\mathbf{x}_k).
\]
Since the $j$-th blocks of both sides are identical for all $j$, the assertion follows.
\end{proof}

\begin{proposition}{\cite[Lemma 4.2.10]{HJ91}} \label{producrulema} 
Let $A, B, C, D \in M_d(\mathbb{R})$. Then, the following identity holds:
\[
(A \otimes B)(C \otimes D) = (AC) \otimes (BD).
\]
\end{proposition}
\begin{proof}
\textbf{(LHS)}:
Set $L := (A \otimes B)(C \otimes D)$. We regard $L$ as a block matrix where the $(i, k)$-th block, denoted by $L_{ik}$. Then,
$$
L_{ik} = \sum_{j=1}^{d} (\text{$(i,j)$-th block} \text{ of } A \otimes B) \times (\text{$(j,k)$-th block}  \text{ of } C \otimes D).
$$
The $(i, j)$-th block of $A \otimes B$ is $a_{ij}B$, and the $(j, k)$-th block of $C \otimes D$ is $c_{jk}D$. Thus,
$$
L_{ik} = \sum_{j=1}^{d} (a_{ij}B) (c_{jk}D).
$$
Since scalars $a_{ij}$ and $c_{jk}$ commute with matrices, we can rearrange the terms:
$$
L_{ik} = \sum_{j=1}^{d} a_{ij} c_{jk} (BD) = \left( \sum_{j=1}^{d} a_{ij} c_{jk} \right) BD.
$$
Thus, we have
$$
L_{ik} = (AC)_{ik} BD.
$$
\textbf{(RHS)}:
Now, consider the matrix $R := (AC) \otimes (BD)$. By the definition of the Kronecker product, the $(i, k)$-th block of $R$ is obtained by
\[
R_{ik} = (AC)_{ik} (BD).
\]
Thus, the assertion follows.
\end{proof}

\begin{theorem}{\cite[Corollary 4.4.7]{HJ91}}  \label{mainlapuin}
Let $A \in M_{d}(\mathbb{R})$ be symmetric and positive definite. Let $G \in M_d(\mathbb{R})$ and $\sigma(G) = \{ \lambda_1, \dots, \lambda_d \} \subset \mathbb{C}$ be the set of eigenvalues of $G$ counting multiplicities.  Assume that
\begin{equation} \label{eq:condition}
    \lambda_i + \lambda_j \neq 0 \quad \text{for all } 1 \leq i, j \leq d.
\end{equation}
Then there exists a unique solution $P$ satisfying the Lyapunov equation:
\begin{equation} \label{eqlyap}
    GP + PG^T = -2A.
\end{equation}
In particular, $P$ is given explicitly by:
\begin{equation} \label{vectorialmap}
\text{\rm Vec}(P) = - (I \otimes G + G \otimes I)^{-1} \text{\rm Vec}(2A).
\end{equation}
where   $\text{Vec}$ and $\otimes$ are defined as in Definitions \ref{vectorizdefn}, \ref{tensorproducde}, respectively.
Moreover, $P$ is symmetric and invertible.
\end{theorem}
\begin{proof}
We solve the matrix equation by transforming it into an equivalent linear system in the vector space $\mathbb{R}^{d^2}$. First, apply the vectorization operator to both sides of \eqref{eqlyap}:
\[
\text{Vec}(GP + PG^T) = \text{Vec}(-2A).
\]
By the linearity of the vectorization operator, we have:
\begin{equation*}
\text{Vec}(GP) + \text{Vec}(PG^T) = -\text{Vec}(2A).
\end{equation*}
Rewriting $GP$ as $GPI$ and $PG^T$ as $IPG^T$, we obtain from Proposition \ref{tensform}:
\begin{align*}
    \text{Vec}(GP) &= \text{Vec}(G P I) = (I \otimes G) \text{Vec}(P), \\
    \text{Vec}(PG^T) &= \text{Vec}(I P G^T) = (G \otimes I) \text{Vec}(P).
\end{align*}
Thus, we obtain
\[
\underbrace{(I \otimes G + G \otimes I)}_{=:\mathcal{K}} \text{Vec}(P) = -\text{Vec}(2A).
\]
To ensure a unique solution, $\mathcal{K}$ must be invertible. Using the Schur decomposition (see \cite[Theorem 2.3.1]{HJ13}), there exists a unitary matrix $U$ such that $G = U T U^*$, where $T$ is an upper triangular matrix with the eigenvalues $\lambda_i$ on its diagonal. We substitute $G = U T U^*$ directly into the expression for $\mathcal{K}$. Noting that $I = U I U^*$ and using Proposition \ref{producrulema}, 
\begin{align*}
\mathcal{K} &= (U I U^*) \otimes (U T U^*) + (U T U^*) \otimes (U I U^*) \\
&= (U \otimes U)(I \otimes T)(U^* \otimes U^*) + (U \otimes U)(T \otimes I)(U^* \otimes U^*) \\
&= (U \otimes U) \underbrace{\left[ (I \otimes T) + (T \otimes I) \right]}_{=:\mathcal{M}} (U \otimes U)^*.
\end{align*}
This shows that $\mathcal{K}$ is similar to $\mathcal{M} = I \otimes T + T \otimes I$. Let us examine the structure of $\mathcal{M}$ to determine its eigenvalues. Observe that
$$
\mathcal{M} = 
\begin{pmatrix}
T & 0 & \dots \\
0 & T & \dots \\
\vdots & \vdots & T
\end{pmatrix}
+
\begin{pmatrix}
\lambda_1 I & * & *\\
0 & \lambda_2 I & * \\
\vdots & \vdots & \lambda_n I
\end{pmatrix}
$$
Since the sum of two upper triangular matrices is also upper triangular, $\mathcal{M}$ is an upper triangular matrix of size $d^2 \times d^2$. Moreover,
$$
\text{diag}(\mathcal{M}) = \{ \lambda_j + \lambda_i\mid 1 \le i, j \le d \}.
$$
Thus, the spectrum of $\mathcal{K}$ is given by:
$$
\sigma(\mathcal{K}) = \{ \lambda_i + \lambda_j \mid 1 \le i, j \le d \}.
$$
By the hypothesis \eqref{eq:condition}, we have $\lambda_i + \lambda_j \neq 0$ for all pairs, which implies $0 \notin \sigma(\mathcal{K})$. Therefore, $\det(\mathcal{K}) \neq 0$, and $\mathcal{K}$ is invertible. Therefore, we obtain the unique explicit solution:
$$
\text{Vec}(P) = -\mathcal{K}^{-1} \text{Vec}(2A).
$$
Next, let us show that $P$ is a symmetric matrix.
First, observe that transposing the original equation $GP + PG^T = -2A$ yields
$$
P^T G^T + G P^T = -2A.
$$
By the uniqueness result, $P$ must coincide with its transpose $P^T$. Hence, $P$ is symmetric. \\
To prove invertibility of $P$, suppose that $P$ is not invertible. Then there exists a non-zero vector $v \in \mathbb{R}^d$ such that $Pv = 0$. Thus, we obtain
\begin{equation*} 
v^T (GP + PG^T) v = -2 v^T A v.
\end{equation*}
By the symmetry of $P$, the left hand side is
$$
v^T G \underbrace{(Pv)}_{=0} + \underbrace{(Pv)^T}_{=0} G^T v = 0.
$$
On the other hand, we have $-2 v^T A v < 0$, which leads to a contradiction.
Therefore, $P$ is invertible, as desired.
\end{proof}

\begin{theorem}\label{explicitconstruction} [Explicit construction]
Let $A \in M_{d}(\mathbb{R})$ be symmetric and positive definite.
Let $G \in M_{d}(\mathbb{R})$. Let $\sigma(G) = \{ \lambda_1, \dots, \lambda_d \} \subset \mathbb{C}$ be the set of eigenvalues of $G$ counting multiplicities. Assume that
\begin{equation*}
    \lambda_i + \lambda_j \neq 0 \quad \text{for all } 1 \leq i, j \leq d.
\end{equation*}
Then, there exists a $d\times d$ symmetric and invertible matrix $S$ such that \eqref{SOHHDlinalg1} holds. In particular, $S$ is explicitly given by
\[
S = (-P)^{-1},
\]
where $P$ is determined by the vectorized equation \eqref{vectorialmap}.
Moreover, 
$$
\mu=\exp\left( \langle Sx, x \rangle  \right) dx
$$
is an infinitesimally invariant measure for $\left(\frac12 \mathrm{trace}(A \nabla^2) + \langle Gx,\nabla \rangle,\; C_0^{\infty}(\mathbb{R}^d)\right)$ and
$$
\tilde Gx=Sx+(G-AS)x
$$
is in fact a SOHHD for $\tilde G x$.
\end{theorem}
\begin{proof}
By Theorem \ref{mainlapuin}, there exists a unique symmetric and invertible matrix $P$ satisfying the Lyapunov equation:
\begin{equation*} 
GP + PG^T = -2A.
\end{equation*}
The explicit vectorized form of $P$ is given by \eqref{vectorialmap}.
Now, define
\[
S = (-P)^{-1}.
\]
This implies $P = -S^{-1}$. Substituting this into the above equation, we have:
\[
G(S^{-1}) + (S^{-1})G^T = 2A.
\]
Thus, we obtain
\[
S(GS^{-1})S + S(S^{-1}G^T)S = 2SAS,
\]
which simplifies to the algebraic Riccati equation:
\begin{equation*} 
SG + G^T S = 2SAS.
\end{equation*}
Multiplying by $S^{-1}$ from the left in the above, we obtain
\[
G+S^{-1}G^{T}S=2AS.
\]
Thus,
\[
\mathrm{trace}(G)+\mathrm{trace}\!\left(S^{-1}G^{T}S\right)=2\,\mathrm{trace}(AS).
\]
Since
\[
\mathrm{trace}\!\left(S^{-1}G^{T}S\right)
=\mathrm{trace}\!\left(G^{T}SS^{-1}\right)
=\mathrm{trace}(G^{T})
=\mathrm{trace}(G),
\]
we get
\begin{equation*}
\,2\mathrm{trace}(G)=\,2\mathrm{trace}(AS).
\end{equation*}
Hence, \eqref{SOHHDlinalg1} holds. The rest follows from Lemma \ref{lemma algebraic riccati}.
\end{proof}

\begin{cor}[uniqueness via stochastic analysis]
Let $A \in M_{d}(\mathbb{R})$ be symmetric and positive definite. Assume that $G \in M_d(\mathbb{R})$ is a Hurwitz matrix. Then, there exists a unique $d \times d$ symmetric matrix
solution $S$ to \eqref{SOHHDlinalg1}.
More precisely, the solution $S$ is real, symmetric and invertible, and has the following explicit form:
$$
S=\left(-\int_{0}^{\infty} e^{Gt}\,(2A)\,e^{G^T t}\,dt \right)^{-1}= (-P)^{-1},
$$
where $P$ is determined by the vectorized equation \eqref{vectorialmap}.
If $\tilde{S}$ is another symmetric matrix solution  to \eqref{SOHHDlinalg1}, then $S=\tilde{S}$.
\end{cor}
\begin{proof}
By Lemma \ref{lemma algebraic riccati}, Proposition \ref{propintegrep}, and Theorem \ref{explicitconstruction}, the existence of a solution to \eqref{SOHHDlinalg1} and its explicit formulas are established. To prove uniqueness, let $\tilde{S}$ be another symmetric matrix solution to \eqref{SOHHDlinalg1}. Then, by Lemma \ref{lemma algebraic riccati}, $e^{\langle \tilde{S}x, x \rangle}\,dx$
is an infinitesimally invariant measure for $(L^{A,G}, C_0^\infty(\mathbb{R}^d))$. By the uniqueness result in Proposition \ref{propintegrep}, there exists a constant $c>0$ such that
\[
c\,e^{\langle \tilde{S}x, x \rangle}=e^{\langle Sx, x \rangle}
\quad \text{for all } x \in \mathbb{R}^d.
\]
Therefore, \(S=\tilde{S}\).
\end{proof}
\begin{lemma}[Schur reordering]\label{reorderSchur}
Let $G$ be a real $d\times d$ matrix and $\Sigma\subset \sigma(G)$ be a complex conjugation closed subset of its spectrum (counting multiplicities). Then, there exists an orthogonal matrix $U$ such that $R:=U^T GU$ is in real Schur form (i.e. quasi upper triangular  with blocks of  size $1\times 1$ corresponding to real eigenvalues and blocks of  size $2\times 2$ corresponding to complex conjugate pairs of eigenvalues on the diagonal) and admits the following block partition 
\begin{eqnarray*}\label{blockformreorder}
R=\begin{pmatrix}R_{11}&R_{12}\\ 0&R_{22}\end{pmatrix},\qquad
\quad \sigma(R_{11})=\Sigma, \quad \sigma(R_{22})=\sigma(G)\setminus \Sigma,
\end{eqnarray*}
where $R_{11}$ and $R_{22}$ are square matrices.
\end{lemma}
\begin{proof}
Direct consequence of the real Schur form in e.g. \cite[Theorem 7.4.1]{GL13}  and diagonal block swapping as in \cite{BD93} (see also \cite[Section 7.6.2]{GL13}).
\end{proof}\\
In order to show that a symmetric solution to the algebraic Riccati equation in the following theorem exists, we adapt and generalize, i.e. streamline a method that was used in \cite[Section 4.]{Liu22} for $A=id$.
\begin{theorem}\label{sol:algRicgen}
Let $A=(a_{ij})_{1\le i,j\le d}$ and  $G=(g_{ij})_{1\le i,j\le d}$ be  matrices of real numbers, and $A$ be symmetric and positive definite. Then there exists a symmetric matrix of real numbers $S=(s_{ij})_{1\le i,j\le d}$, which solves \eqref{SOHHDlinalg1}.
\end{theorem}
\begin{proof}
{\bf Claim 1:} It is enough to show the statement for $A=id$. \\
Indeed, assume that \eqref{SOHHDlinalg1} holds for $A=id$ and let $G$ be arbitrarily given. Since $A$ is symmetric and positive definite, there exists a symmetric positive definite matrix $\sqrt{A}$ with $A=\sqrt{A}\sqrt{A}$. Let $\widehat{G}:=(\sqrt{A})^{-1}G\sqrt{A}$ and note that $(\sqrt{A})^{-1}$ is also symmetric. Since by assumption the statement is true for $A=id$, there exists a symmetric matrix $\widehat{S}$ of real numbers with 
\begin{eqnarray}\label{hat id}
\widehat{S}\widehat{G}+\widehat{G}^T\widehat{S}=2\widehat{S}^2\quad \text{and}\quad  \mathrm{trace}(\widehat{G}-\widehat{S})=0. 
\end{eqnarray}
Let $S:=(\sqrt{A})^{-1} \widehat{S} (\sqrt{A})^{-1}$. Then, $S$ is symmetric. Inserting the expression for $\widehat{G}$ in the first equation of \eqref{hat id}, and multiplying it from the left and from the right with $(\sqrt{A})^{-1}$, one can see that $S$ satisfies $SG+G^TS=2SAS$. Moreover, using the formula $\mathrm{trace}(CD)=\mathrm{trace}(DC)$, we see that $\mathrm{trace}(\widehat{G})=\mathrm{trace}(G)$ and $\mathrm{trace}(AS)=\mathrm{trace}(\widehat{S})$, hence
\[
\mathrm{trace}(G-AS)=\mathrm{trace}(\widehat{G}-\widehat{S})=0
\]
as desired.\\
{\bf Claim 2:} The statement holds for $A=id$.\\
Given the matrix $G$, we first formulate the Schur reordering that will be needed for the proof. We split the spectrum $\sigma(G)$ (counting multiplicities) in two disjoint subsets $\sigma_0(G), \sigma_1(G)$, where $\sigma_1(G)$ is complex conjugation closed and contains no pair summing up to zero (in the sense of Theorem \ref{explicitconstruction}), and $\sigma_0(G)$ is complex conjugation closed and contains $\pm$-pairs, i.e. pairs that sum up to zero (and possibly zero eigenvalues). Then $\sum_{\lambda\in \sigma_0(G)}\lambda=0$, hence
\[
\mathrm{trace}(G)=\sum_{\lambda\in \sigma(G)}\lambda=\sum_{\lambda\in \sigma_1(G)}\lambda.
\]
Now, we apply Lemma \ref{reorderSchur} with $\Sigma=\sigma_0(G)$. Thus  for some orthogonal $U$, we get that
\[
R=U^T G U=\begin{pmatrix}R_{11}&R_{12}\\0&R_{22}\end{pmatrix},\quad
\sigma(R_{11})=\sigma_0(G), \quad \sigma(R_{22})= \sigma_1(G).
\]
By construction, $\sigma(R_{22})\cap(-\sigma(R_{22}))=\sigma_1(G)\cap(-\sigma_1(G))=\varnothing$ and
\[
{\rm {trace}}(R_{22})=\sum_{\lambda\in\sigma_1(G)}\lambda={\rm {trace}}(G).
\]
By Theorem \ref{explicitconstruction} applied to $R_{22}$, we obtain a real symmetric matrix $S_2$ of the same size as $R_{22}$, such that
\[
S_2R_{22}+R_{22}^T S_2=2S_2^2,
\qquad
\mathrm{trace}(S_2)=\mathrm{trace}(R_{22}).
\]
Define the block matrix
\[
\widehat S := \begin{pmatrix}0&0\\ 0&S_2\end{pmatrix},
\]
where $\widehat S$ has the same block partition form as $R$. Then direct multiplication gives
\[
\widehat S R + R^T\widehat S = 2\widehat S^{2}.
\]
Also
\[
\mathrm{trace}(\widehat S)=\mathrm{trace}(S_2)=\mathrm{trace}(R_{22})=\mathrm{trace}(G).
\]
Finally set $S:=U\widehat S U^T$. Then $S$ is symmetric and noting that $G=URU^T$, we see that $S$ satisfies \eqref{SOHHDlinalg1} with $A=id$.
\end{proof}

\section{Appendix}

\subsection{Examples of solutions to algebraic Riccati equations}

\begin{example}
Consider the upper triangular matrix
\[
G=\begin{pmatrix}
1&2&3\\
0&-2&4\\
0&0&-5
\end{pmatrix}.
\]
Its spectrum is $\sigma(G)=\{1,-2,-5\}$, and hence every pairwise sum of eigenvalues is not zero. Since $G$ has a positive eigenvalue, the integral representation  (Proposition \ref{propintegrep}) based on
$$
-\int_0^\infty e^{Gt}(2I)e^{G^Tt}\,dt
$$
does not converge. Nevertheless, \eqref{SOHHDlinalg1} can be explicitly solved by linear algebra. We compute $P$ from the Lyapunov equation
\[
GP+PG^T=-2I.
\]
The two Kronecker products are explicitly
\[
I\otimes G=
\begin{pmatrix}
1&2&3&0&0&0&0&0&0\\
0&-2&4&0&0&0&0&0&0\\
0&0&-5&0&0&0&0&0&0\\
0&0&0&1&2&3&0&0&0\\
0&0&0&0&-2&4&0&0&0\\
0&0&0&0&0&-5&0&0&0\\
0&0&0&0&0&0&1&2&3\\
0&0&0&0&0&0&0&-2&4\\
0&0&0&0&0&0&0&0&-5
\end{pmatrix},
\]
\[
G\otimes I=
\begin{pmatrix}
1&0&0&2&0&0&3&0&0\\
0&1&0&0&2&0&0&3&0\\
0&0&1&0&0&2&0&0&3\\
0&0&0&-2&0&0&4&0&0\\
0&0&0&0&-2&0&0&4&0\\
0&0&0&0&0&-2&0&0&4\\
0&0&0&0&0&0&-5&0&0\\
0&0&0&0&0&0&0&-5&0\\
0&0&0&0&0&0&0&0&-5
\end{pmatrix}.
\]
Hence the $9\times 9$ coefficient matrix $\mathcal K:=I\otimes G+G\otimes I$ is
\[
\mathcal{K}=
\scriptsize
\begin{pmatrix}
 2 & 2 & 3 & 2 & 0 & 0 & 3 & 0 & 0 \\
 0 & -1 & 4 & 0 & 2 & 0 & 0 & 3 & 0 \\
 0 & 0 & -4 & 0 & 0 & 2 & 0 & 0 & 3 \\
 0 & 0 & 0 & -1 & 2 & 3 & 4 & 0 & 0 \\
 0 & 0 & 0 & 0 & -4 & 4 & 0 & 4 & 0 \\
 0 & 0 & 0 & 0 & 0 & -7 & 0 & 0 & 4 \\
 0 & 0 & 0 & 0 & 0 & 0 & -4 & 2 & 3 \\
 0 & 0 & 0 & 0 & 0 & 0 & 0 & -7 & 4 \\
 0 & 0 & 0 & 0 & 0 & 0 & 0 & 0 & -10
\end{pmatrix}.
\]
Since $\mathcal K$ is upper triangular, the linear system
\[
\mathcal K\,\operatorname{Vec}(P)=-2\,\operatorname{Vec}(I)=(-2,0,0,0,-2,0,0,0,-2)^T
\]
can be solved by back-substitution. In particular, one obtains
\[
\operatorname{Vec}(P)
=\frac1{140}\,(-963,\,368,\,29,\,368,\,102,\,16,\,29,\,16,\,28)^T.
\]
Thus,
\[
P=\frac1{140}\begin{pmatrix}
-963&368&29\\
368&102&16\\
29&16&28
\end{pmatrix}.
\]
Defining $S:=(-P)^{-1}$, we obtain
\[
S=\frac1{603995}\begin{pmatrix}
36400&-137760&41020\\
-137760&-389270&365120\\
41020&365120&-3271100
\end{pmatrix}.
\]
One can verify that $S$ satisfies \eqref{SOHHDlinalg1}.
\end{example}

\begin{example}
Consider
\[
G=\begin{pmatrix}-2&1\\4&-2\end{pmatrix}.
\]
Then $\operatorname{tr}(G)=-4\neq0$ and $\det(G)=0$, so $G$ is singular. Its eigenvalues are $\{0,-4\}$.  Applying \eqref{SOHHDlin7} yields
$$
s_{11}=\frac{-4}{25}\big((-2)(-4)-4(-3)\big)=-\frac{16}{5},\qquad
s_{12}=\frac{-4}{25}\big(4(-4)+(-2)(-3)\big)=\frac{8}{5},
$$
and hence $s_{22}=a+d-s_{11}=-4+\frac{16}{5}=-\frac{4}{5}$. Therefore
\[
S=\frac15\begin{pmatrix}-16&8\\8&-4\end{pmatrix},
\]
which indeed satisfies \eqref{SOHHDlinalg1}. On the other hand, let us attempt the vectorization (Lyapunov) method to solve $GP+PG^T=-2I$ with
$P=\begin{pmatrix}x&y\\y&z\end{pmatrix}$. Solving the equation gives
\[
-4x+2y=-2,\qquad 4x-4y+z=0,\qquad 8y-4z=-2.
\]
From the first equation $y=2x-1$, and substituting into the second yields $z=4x-4$. Plugging these into the third equation leads to
\[
4(2x-1)-2(4x-4)=-1
\]
which is a contradiction. Hence the Lyapunov equation has no solution in this case, and the vectorization method fails. Moreover, since $G$ has a zero eigenvalue, the integral representation (Proposition \ref{propintegrep}) based on
\[
-\int_0^\infty e^{Gt}(2I)e^{G^Tt}\,dt
\]
does not converge.
\end{example}

\begin{example}\label{2dimensionac}
Consider
\[
G=\begin{pmatrix}0&-4\\[2pt]4&0\end{pmatrix}.
\]
The eigenvalues of $G$ are
$\{4i,-4i\}$.  Since $a+d=0$, by Proposition~\ref{linearSOHHD}, we have that
$$
S=0
$$
satisfies \eqref{SOHHDlinalg1}. On the other hand, let us attempt the vectorization (Lyapunov) method to solve $GP+PG^T=-2I$ with
$P=\begin{pmatrix}x&y\\y&z\end{pmatrix}$. Using $G^T=-G$, the equation becomes $GP-PG=-2I$. A direct
calculation gives
\[
GP-PG=\begin{pmatrix}-8y&4x-4z\\[2pt]4x-4z&8y\end{pmatrix}=\begin{pmatrix}-2&0\\[2pt]0&-2\end{pmatrix}.
\]
Thus, $-8y=-2$ and $8y=-2$, which is a contradiction.
Hence the vectorization method fails.
Moreover, 
$e^{Gt}(2I)e^{G^Tt}=2I$. Consequently, the integral representation (Proposition~\ref{propintegrep})
based on
\[
-\int_0^\infty e^{Gt}(2I)e^{G^Tt}\,dt
\]
does not converge.
Finally, note that $\langle Gx,x\rangle=0$ for all $x\in\mathbb{R}^2$. Hence, by \cite[Corollary 3.41]{LST22},
the Markov process associated with $\bigl(L^{G},C_0^{\infty}(\mathbb{R}^2)\bigr)$ is recurrent. Therefore,
$dx$ is the unique infinitesimally invariant measure for $\big(L^{G},C_0^{\infty}(\mathbb{R}^2)\big)$
(up to a multiplicative constant) by \cite[Theorem 3.15]{LT22}. By Theorem \ref{nonexofOHHD}(ii), with $\widetilde{\mathbf{G}}=G$, $\Phi=0$, and $\mathbf{B}=G$,
it follows that $Gx=0+Gx$ is the unique orthogonal Helmholtz--Hodge decomposition of $Gx$.
\end{example}

\begin{example}
Consider the matrices
\[
G = \begin{pmatrix} 0 & 1 \\ -2 & -3 \end{pmatrix}, \qquad A = \begin{pmatrix} 2 & 1 \\ 1 & 1 \end{pmatrix}.
\]
The spectrum of $G$ is $\sigma(G) = \{-1, -2\}$, which implies that $G$ is Hurwitz. The matrix $A$ is symmetric positive definite with $\det(A)=1$. Let us find a symmetric matrix $S$ satisfying
\begin{equation*} 
SG + G^T S = 2SAS, \qquad \mathrm{trace}(AS) = \mathrm{trace}(G).
\end{equation*}
By setting $P = (-S)^{-1}$, let us find a symmetric matrix $P$ satisfying
\begin{equation*}
GP + PG^T = -2A.
\end{equation*}
We present three distinct approaches to find $S$, which all gives the same answers.  \\ \\
\noindent \textbf{Method 1: Explicit Formula.}
Note that employing the notation from Theorem \ref{expliformuouoper}, we have $a=0, b=1, c=-2, d=-3$ for $G$ and $\alpha=2, \beta=1, \gamma=1$ for $A$. Then, we have
\begin{align*}
&\alpha \gamma - \beta^2=1, \qquad a+d=-3 \\
&\gamma b - \alpha c + \beta(a-d) = 1(1) - 2(-2) + 1(0 - (-3)) = 1 + 4 + 3 = 8 \neq 0.
\end{align*}

Thus, by Theorem \ref{expliformuouoper}(i),
\begin{align*}
s_{11} &= -\frac{3}{73} \Big( \gamma a(a+d) + \alpha c^2 - \gamma bc - 2\beta ac \Big) = -\frac{3}{73}(10) = -\frac{30}{73}, \\
s_{12} &= -\frac{3}{73} \Big( \alpha cd + \gamma ab - 2\beta ad \Big) = -\frac{3}{73}(12) = -\frac{36}{73}, \\
s_{22} &= -\frac{3}{73} \Big( \alpha d(a+d) + \gamma b^2 - \alpha bc - 2\beta bd \Big) = -\frac{3}{73}(29) = -\frac{87}{73}.
\end{align*}
Thus, the matrix $S$ is given by:
\[
S = -\frac{3}{73} \begin{pmatrix} 10 & 12 \\ 12 & 29 \end{pmatrix}.
\]

\vspace{1em}
\noindent \textbf{Method 2: Vectorization.}
First, we compute the Kronecker products. For $I \otimes G$:
\[
I \otimes G = \begin{pmatrix} 1 & 0 \\ 0 & 1 \end{pmatrix} \otimes G = 
\begin{pmatrix} 
0 & 1 & 0 & 0 \\ 
-2 & -3 & 0 & 0 \\ 
0 & 0 & 0 & 1 \\ 
0 & 0 & -2 & -3 
\end{pmatrix}.
\]
For $G \otimes I$:
\[
G \otimes I = \begin{pmatrix} 0 & 1 \\ -2 & -3 \end{pmatrix} \otimes I = 
\begin{pmatrix} 
0 & 0 & 1 & 0 \\ 
0 & 0 & 0 & 1 \\ 
-2 & 0 & -3 & 0 \\ 
0 & -2 & 0 & -3 
\end{pmatrix}.
\]
Summing these gives the system matrix $\mathcal{K}$:
\[
\mathcal{K} := (I \otimes G) + (G \otimes I) = 
\begin{pmatrix} 
0 & 1 & 1 & 0 \\ 
-2 & -3 & 0 & 1 \\ 
-2 & 0 & -3 & 1 \\ 
0 & -2 & -2 & -6 
\end{pmatrix}.
\]
Let $\operatorname{Vec}(P) = (p_{11}, p_{21}, p_{12}, p_{22})^T$. Since $P$ is symmetric, $p_{21} = p_{12}$. The vectorized equation $\mathcal{K}\operatorname{Vec}(P) = \operatorname{Vec}(-2A)$ becomes:
\[
\begin{pmatrix} 0 & 1 & 1 & 0 \\ -2 & -3 & 0 & 1 \\ -2 & 0 & -3 & 1 \\ 0 & -2 & -2 & -6 \end{pmatrix}
\begin{pmatrix} p_{11} \\ p_{12} \\ p_{12} \\ p_{22} \end{pmatrix}
=
\begin{pmatrix} -4 \\ -2 \\ -2 \\ -2 \end{pmatrix}.
\]
Then,
\[
P = \begin{pmatrix} p_{11} & p_{12} \\ p_{21} & p_{22} \end{pmatrix} = \begin{pmatrix} 29/6 & -2 \\ -2 & 5/3 \end{pmatrix} = \frac{1}{6}\begin{pmatrix} 29 & -12 \\ -12 & 10 \end{pmatrix}.
\]
Recovering $S$:
\[
S = -P^{-1} = - 6 \left( \begin{pmatrix} 29 & -12 \\ -12 & 10 \end{pmatrix} \right)^{-1} 
= - 6 \cdot \frac{1}{146} \begin{pmatrix} 10 & 12 \\ 12 & 29 \end{pmatrix} 
= -\frac{3}{73} \begin{pmatrix} 10 & 12 \\ 12 & 29 \end{pmatrix}.
\]

\vspace{1em}
\noindent \textbf{Method 3: Integral Representation.}
Since $G$ is Hurwitz, the unique solution $P$ to the Lyapunov equation $GP+PG^T = -2A$ has the following integral formula
\[
P = \int_0^\infty e^{Gt}(2A)e^{G^T t} \, dt.
\]
First, we compute the matrix exponential $e^{Gt}$. We diagonalize $G$ as $G=VDV^{-1}$:
\[
V = \begin{pmatrix} 1 & 1 \\ -1 & -2 \end{pmatrix}, \quad D = \begin{pmatrix} -1 & 0 \\ 0 & -2 \end{pmatrix}, \quad V^{-1} = \begin{pmatrix} 2 & 1 \\ -1 & -1 \end{pmatrix}.
\]
\[
e^{Gt} = V e^{Dt} V^{-1} = \begin{pmatrix} 1 & 1 \\ -1 & -2 \end{pmatrix} \begin{pmatrix} e^{-t} & 0 \\ 0 & e^{-2t} \end{pmatrix} \begin{pmatrix} 2 & 1 \\ -1 & -1 \end{pmatrix} = \begin{pmatrix} 2e^{-t}-e^{-2t} & e^{-t}-e^{-2t} \\ -2e^{-t}+2e^{-2t} & -e^{-t}+2e^{-2t} \end{pmatrix}.
\]
Then,
\begin{align*}
&M(t) = e^{Gt}(2A)e^{G^T t}=\begin{pmatrix}  m_{11}(t)& m_{12}(t) \\ m_{12}(t) & m_{22}(t) \end{pmatrix}\\
&=\begin{pmatrix}  26e^{-2t}-32e^{-3t}+10e^{-4t}& -26e^{-2t}+48e^{-3t}-20e^{-4t} \\ -26e^{-2t}+48e^{-3t}-20e^{-4t} & 26e^{-2t}-64e^{-3t}+40e^{-4t} \end{pmatrix}
\end{align*}
We integrate each component from $0$ to $\infty$:
\[
\int_0^\infty m_{11}(t) dt = \left[ -13e^{-2t} + \frac{32}{3}e^{-3t} - \frac{5}{2}e^{-4t} \right]_0^\infty = (0) - (-13 + \frac{32}{3} - \frac{5}{2}) = \frac{29}{6}
\]
\[
\int_0^\infty m_{12}(t) dt = \int_0^\infty m_{21}(t) dt = \left[ 13e^{-2t} - 16e^{-3t} + 5e^{-4t} \right]_0^\infty = (0) - (13 - 16 + 5) = -2
\]
\[
\int_0^\infty m_{22}(t) dt = \left[ -13e^{-2t} + \frac{64}{3}e^{-3t} - 10e^{-4t} \right]_0^\infty = (0) - (-13 + \frac{64}{3} - 10) = \frac{5}{3}.
\]
Thus, the integral yields:
\[
P = \begin{pmatrix} p_{11} & p_{12} \\ p_{21} & p_{22} \end{pmatrix} = \frac{1}{6}\begin{pmatrix} 29 & -12 \\ -12 & 10 \end{pmatrix}.
\]
Finally, we recover $S$:
\[
S = -P^{-1} = -\frac{3}{73} \begin{pmatrix} 10 & 12 \\ 12 & 29 \end{pmatrix}.
\]
All three methods yield the same result. Finally, we verify the trace constraint $\mathrm{trace}(AS)=\mathrm{trace}(G)$:
\[
\mathrm{trace}(AS) = \mathrm{trace}\left( -\frac{3}{73} \begin{pmatrix} 2 & 1 \\ 1 & 1 \end{pmatrix} \begin{pmatrix} 10 & 12 \\ 12 & 29 \end{pmatrix} \right) 
=-\frac{3}{73} \mathrm{trace}\begin{pmatrix} 32 & 53 \\ 22 & 41 \end{pmatrix} = -\frac{3}{73}(73) = -3 = \mathrm{trace}(G),
\]
as desired.
\end{example}

\subsection{Alternative proof of existence of an explicit  symmetric solution $S$ to \eqref{SOHHDlin5} in case $d=2$}
Let up to this end $d=2$. In order to show that there is a solution to \eqref{SOHHDlin5}, we first start with the case $A=id$. Let
\begin{eqnarray*}
G=\begin{pmatrix} g_{11}  & g_{12} \\  g_{21}  & g_{22} \end{pmatrix}=\begin{pmatrix} a & b \\  c  & d \end{pmatrix}.
\end{eqnarray*}
Then 
\begin{eqnarray*}
H=G-S=\begin{pmatrix} a-s_{11}  &b- s_{12} \\c-  s_{21}  & d- s_{22} \end{pmatrix},\quad S=\begin{pmatrix}s_{11}&s_{12}\\ s_{12}&s_{22}\end{pmatrix}.
\end{eqnarray*}
Using the symmetry of $S$ and $h_{11}+h_{22}=(a-s_{11}) + (d-s_{22})=0$ for the second to last equations right below, \eqref{SOHHDlin5} is equivalent to 
\begin{align}
& s_{11}(a-s_{11})+s_{12}(c-s_{12})=0, \label{eqq1}\\
& s_{12}(b-s_{12})+s_{22}(d-s_{22})=0, \label{eqq2}\\
& s_{11}(b-s_{12})+s_{22}(c-s_{12})=0, \label{eqq3}\\
& s_{22}=a+d-s_{11}. \label{eqq4}
\end{align}
Now there are three cases:\\[3pt]
\textbf{1)} $b=c$. in this case $G$ is symmetric. Choosing $S:=G$ gives $H=G-S=0$, and then $SH+H^TS=0$ and
$\mathrm{trace}(H)=0$ hold. Thus \eqref{SOHHDlinalg1} is satisfied.\\[3pt]
\textbf{2)} $a+d=0$ (i.e.\ $\mathrm{trace}(G)=0$).
Choosing $S:=0$ gives $H=G$. Then $\mathrm{trace}(H)=\mathrm{trace}(G)=0$, and $SH+H^TS=0$ is trivial since $S=0$. Thus \eqref{SOHHDlinalg1} holds.\\[3pt]
\textbf{3)} $b\neq c$ and $a+d\neq0$. In this case set
\[
t:=a+d\neq0,\qquad u:=b-c\neq0.
\]
The trace condition \eqref{eqq4} gives
\[
s_{22}=a+d-s_{11}=t-s_{11}.
\]
We simplify \eqref{eqq3} using $s_{22}=t-s_{11}$:
\begin{align*}
0
&=s_{11}(b-s_{12})+(t-s_{11})(c-s_{12})
= s_{11}(b-c)+t(c-s_{12}).
\end{align*}
Since $t\neq0$ and $u=b-c\neq0$, this is equivalent to
\begin{equation}\label{eq:linear_rel}
t(c-s_{12})=-u\,s_{11}
\qquad\Longleftrightarrow\qquad
t s_{12}=t c+u s_{11}.
\end{equation}
Next, subtract \eqref{eqq1} from \eqref{eqq2}. Using $s_{22}=t-s_{11}$ and $d-s_{22}=s_{11}-a$ from \eqref{eqq4}, we obtain
\begin{align*}
0
&=\big(s_{12}(b-s_{12})+s_{22}(d-s_{22})\big)
 -\big(s_{11}(a-s_{11})+s_{12}(c-s_{12})\big)\\
&= s_{12}(b-c) + s_{22}(d-s_{22}) - s_{11}(a-s_{11})\\
&= u s_{12} + (t-s_{11})(s_{11}-a)-s_{11}(a-s_{11})
= u s_{12} + t(s_{11}-a).
\end{align*}
Therefore,
\begin{equation}\label{eqlinearrel2}
t s_{11} + u s_{12}=a t.
\end{equation}
Combining \eqref{eq:linear_rel} and \eqref{eqlinearrel2}, we obtain a $2\times2$ linear system for
$(s_{11},s_{12})$:
\[
\begin{pmatrix}
t & u\\
u & -t
\end{pmatrix}
\binom{s_{11}}{s_{12}}
=
\binom{a t}{-c t}.
\]
Since $\det\begin{pmatrix}t&u\\ u&-t\end{pmatrix}=-(t^2+u^2)\neq0$, this system has a unique solution.
Solving it gives
\[
s_{11}=\frac{t(at-cu)}{t^2+u^2},
\qquad
s_{12}=\frac{t(ct+au)}{t^2+u^2}.
\]
Since $t=a+d$ and $u=b-c$, we conclude
\begin{equation}\label{SOHHDlin7}
\begin{cases}
\ s_{11}= \frac{a+d}{(a+d)^2+(b-c)^2}\big (a(a+d)-c(b-c)\big ), \\[6pt]
\  s_{12}=\frac{a+d}{(a+d)^2+(b-c)^2}\big (c(a+d)+a(b-c)\big ), \\[6pt]
\ s_{22}=a+d-s_{11}.
\end{cases}
\end{equation}
\begin{theorem} \label{expliformuouoper}
Let
\[
G=\begin{pmatrix}a&b\\ c&d\end{pmatrix},\qquad
A=\begin{pmatrix}\alpha&\beta\\ \beta&\gamma\end{pmatrix}
\quad(\alpha>0,\ \delta:=\det A=\alpha\gamma-\beta^{2}>0).
\]
Then, there exists a symmetric matrix $S=\begin{pmatrix}s_{11}& s_{12}\\ s_{12}& s_{22}\end{pmatrix}$ satisfying \eqref{SOHHDlinalg1}, and given as follows:
\begin{itemize}
\item[(i)]
If $\gamma b -\alpha c + \beta(a-d) =0$, then $S = A^{-1}G$, and in particular,
\[
s_{11} = \frac{\gamma a - \beta c}{\alpha\gamma - \beta^2}, \qquad
s_{12} = \frac{\gamma b - \beta d}{\alpha\gamma - \beta^2}, \qquad
s_{22} = \frac{\alpha d - \beta b}{\alpha\gamma - \beta^2}.
\]

\item[(ii)]
If $a+d=0$, then $S=0$, that is,
\[
s_{11} = 0, \qquad s_{12} = 0, \qquad s_{22} = 0.
\]

\item[(iii)] If $a+d \neq 0$ and $\gamma b- \alpha c + \beta(a-d) \neq 0$, then
\begin{align*}
s_{11} &= \frac{a+d}{ (\alpha \gamma-\beta^2) (a+d)^2 +\big( \gamma b - \alpha c + \beta(a-d)\big)^2} \Big( \gamma a(a+d) + \alpha c^2 - \gamma bc - 2\beta ac \Big), \\[8pt]
s_{12} &= \frac{a+d}{ (\alpha \gamma-\beta^2) (a+d)^2 +\big( \gamma b - \alpha c + \beta(a-d)\big)^2} \Big( \alpha cd + \gamma ab - 2\beta ad \Big), \\[8pt]
s_{22} &= \frac{a+d}{ (\alpha \gamma-\beta^2) (a+d)^2 +\big(\gamma b - \alpha c + \beta(a-d)\big)^2} \Big(\alpha d(a+d) + \gamma b^2 - \alpha bc - 2\beta bd \Big).
\end{align*}
\end{itemize}
\end{theorem}
\begin{proof}
\textbf{Step 1 (A transformation of \eqref{SOHHDlinalg1} via Cholesky decomposition for $A$)}\\
Using the Cholesky decomposition, consider the lower triangular matrix
\begin{equation} \label{expressilma}
L:=
\begin{pmatrix}
\sqrt{\alpha}&0\\[2pt]
\dfrac{\beta}{\sqrt{\alpha}}&\dfrac{\sqrt{\delta}}{\sqrt{\alpha}}
\end{pmatrix},
\end{equation}
satisfying $A=LL^{T}$. Note that its inverse is given by
\begin{equation} \label{expressilmainv}
L^{-1}=
\begin{pmatrix}
\dfrac{1}{\sqrt{\alpha}}&0\\[8pt]
-\dfrac{\beta}{\sqrt{\alpha}\sqrt{\delta}}&\dfrac{\sqrt{\alpha}}{\sqrt{\delta}}
\end{pmatrix}.
\end{equation}
Define the transformed matrix
\begin{equation} \label{hatgcalcu}
\widehat{G}:=L^{-1}GL.
\end{equation}
\begin{lemma}\label{recover S hat}
There exists a symmetric matrix $S$ solving \eqref{SOHHDlinalg1}, if and only if there exists a symmetric matrix $\widehat S$ solving
\begin{equation}\label{eqmainhat}
\widehat S\,\widehat G+\widehat G^{T}\widehat S = 2\widehat S^{2},
\qquad
\mathrm{trace}(\widehat S)=\mathrm{trace}(\widehat G).
\end{equation}
\end{lemma}
\begin{proof}
Assume there exists a symmetric matrix $S$ satisfying \eqref{SOHHDlinalg1}. Let
\[
\widehat{S}:=L^{T}SL.
\]
Then $\widehat S$ is symmetric. Moreover, using the representations for $\widehat S$, $\widehat G$, and $A$, we get
\begin{equation} \label{isomorphieq1}
\widehat S\,\widehat G+\widehat G^{T}\widehat{S}
= L^{T}(SG+G^{T}S)L, \quad
2\widehat S^{2}=2(L^{T}SL)(L^{T}SL)=2L^{T}(SAS)L,
\end{equation}
and
\begin{align} 
&\mathrm{trace}(\widehat G)=\mathrm{trace}(L^{-1}GL)=\mathrm{trace}(G), \nonumber \\
&\mathrm{trace}(\widehat{S})=\mathrm{trace}(L^{T}SL)=\mathrm{trace}(SLL^{T})  =\mathrm{trace}(SA)=\mathrm{trace}(AS). \label{isomorphieq2}
\end{align}
Thus $\widehat S$ satisfies \eqref{eqmainhat}. Conversely, suppose there exists a symmetric matrix $\widehat S$ such that \eqref{eqmainhat} holds. Setting
$$
S:=(L^{T})^{-1}\widehat{S} L^{-1},
$$
by \eqref{isomorphieq1} and \eqref{isomorphieq2}, we recover \eqref{SOHHDlinalg1}.
\end{proof}\\
\textbf{Step 2 (Computation of $\widehat{G}$).}
A direct computation, using \eqref{expressilma}, \eqref{expressilmainv}, \eqref{hatgcalcu}, yields
\begin{equation} \label{informathatab}
\widehat G:=
\begin{pmatrix}\widehat a&\widehat b\\ \widehat c&\widehat d\end{pmatrix}
=
\begin{pmatrix}
a+\dfrac{b\beta}{\alpha} & \dfrac{\sqrt{\delta} b}{\alpha}\\[10pt]
\dfrac{\alpha^{2}c+\alpha\beta(d-a)-\beta^{2}b}{\alpha\sqrt{\delta}} & d-\dfrac{\beta b}{\alpha}
\end{pmatrix},
\qquad
\mathrm{trace}(\widehat G)=a+d.
\end{equation}
\medskip
\noindent\textbf{Step 3 (Applying Suda's formula for $\widehat S$ (see \cite[Proposition 1]{Sud2})).}
Let
\[
\widehat S=\begin{pmatrix}x&y\\ y&z\end{pmatrix}.
\]
Consider the following cases:\\[3pt]
\emph{(i) If $\widehat b=\widehat c$} (i.e., $\gamma b -\alpha c + \beta(a-d) = 0$), then $\widehat G$ is symmetric. In this case, we may choose
\[
\widehat S=\widehat G,
\]
which implies
$$
S=(L^{-1})^T L^{-1} G = (L L^T)^{-1} G =A^{-1}G= \frac{1}{\alpha\gamma - \beta^2}
\begin{pmatrix}
\gamma a - \beta c & \gamma b - \beta d \\
\gamma b - \beta d & \alpha d - \beta b
\end{pmatrix}.
$$

\noindent\emph{(ii) If $\widehat a+\widehat d=0$} (i.e., $a+d=0$), then
\[
\widehat S=0,
\]
and hence $S=0$.

\noindent\emph{(iii) If $\widehat{b} \neq \widehat{c}$ and $\widehat a+\widehat d \neq 0$} (i.e., $\gamma b - \alpha c + \beta(a-d) \neq 0$ and $a+d \neq 0$), then using the formula \eqref{SOHHDlin7} (derived in \cite[Proposition 1]{Sud2}) the unique symmetric solution of \eqref{eqmainhat} is given by
\begin{equation}\label{eqhatsuda}
\begin{cases}
\displaystyle
x=\frac{\widehat a+\widehat d}{(\widehat a+\widehat d)^{2}+(\widehat b-\widehat c)^{2}}\Bigl(\widehat a\,(\widehat a+\widehat d)-\widehat c\,(\widehat b-\widehat c)\Bigr),\\[10pt]
\displaystyle
y=\frac{\widehat a+\widehat d}{(\widehat a+\widehat d)^{2}+(\widehat b-\widehat c)^{2}}\Bigl(\widehat c\,(\widehat a+\widehat d)+\widehat a\,(\widehat b-\widehat c)\Bigr),\\[10pt]
\displaystyle
z=\widehat a+\widehat d-x.
\end{cases}
\end{equation}
In particular, this $\widehat S$ satisfies \eqref{eqmainhat}.

\medskip
\noindent\textbf{Step 4 (Recovering $S$ from $\widehat{S}$).}
The solution $S$ of \eqref{SOHHDlinalg1} is obtained via $\widehat{S}$ of Step 3(iii) through the formula derived at the end of the proof of Lemma \ref{recover S hat}
\[
S=L^{-T}\widehat S L^{-1}=(L^{-1})^{T}\widehat S L^{-1}.
\]
We have
\[
L^{-1}=
\begin{pmatrix}
p&0\\ q&r
\end{pmatrix}
\quad\text{with}\quad
p=\frac{1}{\sqrt{\alpha}},\qquad
q=-\frac{\beta}{\sqrt{\alpha}\sqrt{\delta}},\qquad
r=\frac{\sqrt{\alpha}}{\sqrt{\delta}}.
\]
A straightforward multiplication gives the entries of $S$ in terms of $(x,y,z)$:
\begin{equation}\label{eqfromhat}
\begin{aligned}
s_{11}&=p^{2}x+2pq\,y+q^{2}z
=\frac{x}{\alpha}-\frac{2\beta}{\alpha\sqrt{\delta}}\,y+\frac{\beta^{2}}{\alpha\delta}\,z,\\[6pt]
s_{12}&=r(py+qz)
=\frac{y}{\sqrt{\delta}}-\frac{\beta}{\delta}\,z,\\[6pt]
s_{22}&=r^{2}z=\frac{\alpha}{\delta}\,z,
\end{aligned}
\end{equation}
where $x,y,z$ are defined as in \eqref{eqhatsuda}.
Then, as in the argument of {\bf Step 1},
$S$ fulfills \eqref{SOHHDlinalg1}.\\ \\
\noindent\textbf{Step 5 (Explicit expression for $x,y,z$).}
First, observe that
\begin{align*}
\widehat{b} - \widehat{c} &= \frac{b\sqrt{\delta}}{\alpha} - \frac{\alpha^2 c + \alpha\beta(d-a) - \beta^2 b}{\alpha\sqrt{\delta}} = \frac{b\delta - \alpha^2 c - \alpha\beta(d-a) + \beta^2 b}{\alpha\sqrt{\delta}} \\[10pt]
&= \frac{b(\alpha\gamma - \beta^2) - \alpha^2 c - \alpha\beta d + \alpha\beta a + \beta^2 b}{\alpha\sqrt{\delta}} = \frac{\gamma b - \alpha c - \beta d + \beta a}{\sqrt{\delta}} \\[10pt]
&= \frac{\gamma b - \alpha c + \beta(a-d)}{\sqrt{\delta}}.
\end{align*}
Let
\begin{equation} \label{definitiop}
P:=\gamma b - \alpha c + \beta(a-d). 
\end{equation}
Then, 
$$
\widehat{b}-\widehat{c}=\frac{P}{\sqrt{\delta}}.
$$
Now, note that
\begin{align*}
(\widehat{a}+\widehat{d})^2 + (\widehat{b}-\widehat{c})^2 &= (a+d)^2 + \frac{P^2}{\delta} = \frac{\delta(a+d)^2 + P^2}{\delta}.
\end{align*}
Let
\begin{equation} \label{definitionk}
K: = \delta(a+d)^2 + P^2.
\end{equation}
Then,
$$
(\widehat{a}+\widehat{d})^2 + (\widehat{b}-\widehat{c})^2 =\frac{K}{\delta}.
$$
Using the formula for $x$ in \eqref{eqhatsuda} and substituting \eqref{informathatab}, we get
\begin{align*} 
x &= \frac{\delta(a+d)}{K} \left[ \left( \frac{\alpha a + b\beta}{\alpha} \right)(a+d) - \frac{(\alpha^{2}c+\alpha\beta(d-a)-\beta^{2}b)}{\alpha\sqrt{\delta}} \cdot \left( \frac{P}{\sqrt{\delta}} \right) \right] \nonumber \\
&= \frac{(a+d)}{\alpha K} \left[ \delta(\alpha a + b\beta)(a+d) -C P \right],  
\end{align*}
where we define
$$
C:=\alpha^2 c + \alpha\beta(d-a) - \beta^2 b.
$$
Similarly, for $y$, using \eqref{eqhatsuda} and \eqref{informathatab},
\begin{align}
y &= \frac{\delta(a+d)}{K} \left[ \frac{(\alpha^{2}c+\alpha\beta(d-a)-\beta^{2}b)}{\alpha\sqrt{\delta}}(a+d) + \left( \frac{\alpha a + b\beta}{\alpha} \right) \left( \frac{P}{\sqrt{\delta}} \right) \right] \nonumber \\
&=\frac{\sqrt{\delta}(a+d)}{\alpha K} \bigg[ C(a+d) + (\alpha a + b\beta)P \bigg]. \label{explicitformuy}
\end{align}
Let $\tau = a+d$. Then $z = \tau - x$.
From the formula for $z$ in \eqref{eqhatsuda} and \eqref{informathatab},
\begin{align} \label{formulaforz}
z = \tau - \frac{\tau}{\alpha K} \Bigl[ \delta(\alpha a + b\beta)\tau - CP \Bigr] = \frac{\tau}{\alpha K} \bigg[ \underbrace{\alpha K - \Bigl( \delta(\alpha a + b\beta)\tau - CP \Bigr)}_{=: N} \bigg].
\end{align}
Substituting $K = \delta\tau^2 + P^2$ into the expression for $N$, we obtain
\begin{align*}
N &= \alpha(\delta\tau^2 + P^2) - \delta(\alpha a + b\beta)\tau + CP = \alpha\delta\tau^2 + \alpha P^2 - \delta(\alpha a + b\beta)\tau + CP \\[10pt]
&= \delta\tau \big ( \alpha\tau - (\alpha a + b\beta) \big ) + P \big ( \alpha P + C \big ).
\end{align*}
Observe that
\[
\alpha\tau - \alpha a - b\beta = \alpha(a+d) - \alpha a - b\beta = \alpha a + \alpha d - \alpha a - b\beta = \alpha d - b\beta,
\]
and that
\begin{align*}
\alpha P + C &= \alpha(b\gamma - c\alpha + \beta a - \beta d) + (\alpha^2 c + \alpha\beta d - \alpha\beta a - \beta^2 b) \\
&= \alpha b\gamma - \alpha^2 c + \alpha\beta a - \alpha\beta d + \alpha^2 c + \alpha\beta d - \alpha\beta a - \beta^2 b \\
&=\alpha b\gamma - \beta^2 b = b(\alpha\gamma - \beta^2) =\delta b.
\end{align*}
Therefore,
\[
N = \delta\tau(\alpha d - b\beta) + Pb\delta = \delta \Big( \tau(\alpha d - b\beta) + bP \Big),
\]
where
\begin{align*}
\tau(\alpha d - b\beta) &= (a+d)(\alpha d - b\beta) = \alpha ad - \beta ab + \alpha d^2 - \beta  bd, \\
bP &= b(b\gamma - c\alpha + \beta a - \beta d) =\gamma b^2 -\alpha bc +  \beta ab - \beta b d.
\end{align*}
Thus,
\begin{align*}
N &= \delta \Big( \alpha ad + \alpha d^2 + b^2\gamma - bc\alpha - 2\beta bd \Big)\\
&= \delta \Big( \alpha d(a+d) + \gamma b^2 - \alpha bc - 2\beta bd \Big).
\end{align*}
It follows from \eqref{formulaforz} that 
\begin{align}
z&=\frac{\tau}{\alpha K} \cdot N=
 \frac{a+d}{\alpha K} \cdot \delta \Big( \alpha d(a+d) + \gamma b^2 - \alpha bc - 2\beta bd \Big) \nonumber \\
 &= \frac{(a+d)(\alpha\gamma - \beta^2)}{\alpha K} \Big( \alpha d(a+d) + \gamma b^2 - \alpha bc - 2\beta bd \Big). \label{explicforz}
\end{align}

\noindent\textbf{Step 6 (Explicit expression for $s_{22}, s_{12}, s_{11}$).} \\
Using the expression for $s_{22}$ in \eqref{eqfromhat} and \eqref{explicforz}, we have
\begin{align}
s_{22} &= \frac{\alpha}{\delta} \times z = \frac{\alpha}{\delta} \times \left[ \frac{\tau \delta}{\alpha K} \Big( \alpha d(a+d) + \gamma b^2 - \alpha bc - 2\beta bd \Big) \right] \nonumber \\
&= \frac{\tau}{K} \Big( \alpha d(a+d) + \gamma b^2 - \alpha bc - 2\beta bd \Big). \label{explicitformuz}
\end{align}
Next, we compute $s_{12}$. From \eqref{eqfromhat}, 
$$
s_{12} = \frac{y}{\sqrt{\delta}} - \frac{\beta}{\delta}z.
$$
First, by \eqref{explicitformuy},
\[
\frac{y}{\sqrt{\delta}} = \frac{1}{\sqrt{\delta}} \left( \frac{\tau\sqrt{\delta}}{\alpha K} \Big( C(a+d) + (\alpha a + b\beta)P \Big)\right) = \frac{\tau}{\alpha K} \Big(  \underbrace{C(a+d) + (\alpha a + b\beta)P}_{=:Y_{num}} \Big).
\]
And by \eqref{explicforz},
\[
\frac{\beta}{\delta}z = \frac{\beta}{\delta} \left( \frac{\tau\delta}{\alpha K} \Big( \alpha d(a+d) + \gamma b^2 - \alpha bc - 2\beta bd\Big) \right) = \frac{\tau\beta}{\alpha K} \Big( \underbrace{\alpha d(a+d) + \gamma b^2 - \alpha bc - 2\beta bd}_{=:Z_{num}}\Big).
\]
Combining these, we factor out the common term $\frac{\tau}{\alpha K}$:
\[
s_{12} = \frac{\tau}{\alpha K} \Big( Y_{num} - \beta Z_{num} \Big).
\]
We now calculate the numerators within the brackets:
\begin{align*}
Y_{num} &= (\alpha^2 c + \alpha\beta d - \alpha\beta a - \beta^2 b)\tau + (\alpha a + b\beta)P, \\
\beta Z_{num} &= \beta \Big( \alpha d \tau +b\big( \gamma b-\alpha c  - 2 \beta d\big) \Big)\\
&= \beta \Big( \alpha d \tau + b\big( \gamma b -\alpha c  + \beta a - \beta d	\big) - \beta b d - \beta b a \Big)\\
&= \beta \Big( (\alpha d -\beta b) \tau +bP \Big) = (\alpha\beta d -\beta^2 b)\tau + \beta b P.
\end{align*}
Subtracting these terms yields:
\begin{align*}
 Y_{num} - \beta Z_{num} &= \tau \Bigl( \alpha^2 c + \alpha\beta d - \alpha\beta a - \beta^2 b - \alpha\beta d + \beta^2 b \Bigr)  + P \Bigl( \alpha a + b\beta -\beta b \Bigr) \\
&= \tau (\alpha^2 c - \alpha\beta a) + P(\alpha a)=\alpha \Bigl[ (\alpha c - \beta a)\tau + a P \Bigr].
\end{align*}
Substituting this back into the expression for $s_{12}$, we obtain
\begin{align}
s_{12} & = \frac{\tau}{\alpha K} \cdot \alpha \Big[ (\alpha c - \beta a)(a+d) + a(b\gamma - c\alpha + \beta a - \beta d)  \Big] \nonumber \\
&= \frac{\tau}{K} \Big(\alpha a c+ \alpha c d - \beta a^2 - \beta a d + \gamma a b - \alpha a c + \beta a^2- \beta a d \Big) \nonumber \\
&= \frac{\tau}{K} \Big( \alpha c d + \gamma a b  - 2 \beta a d \Big). \label{explicitformfi}
\end{align}
Thus, we obtain the final explicit form for $s_{12}$:
\[
s_{12} = \frac{\tau}{K} \Bigl[ \alpha c d + \gamma a b - 2 \beta a d \Bigr].
\]
Finally, we determine $s_{11}$ using the trace constraint. Recall that $\mathrm{trace}(AS) = \mathrm{trace}(G) = \tau$.
Given $A = \begin{pmatrix} \alpha & \beta \\ \beta & \gamma \end{pmatrix}$ and $S = \begin{pmatrix} s_{11} & s_{12} \\ s_{12} & s_{22} \end{pmatrix}$, the trace condition becomes
\[
\mathrm{trace}(AS) = \alpha s_{11} + 2\beta s_{12} + \gamma s_{22} = \tau.
\]
Rearranging this to find explicitly $s_{11}$ and using \eqref{explicitformfi} and \eqref{explicitformuz}, we get
\begin{align}
s_{11} &= \frac{1}{\alpha} \Bigl( \tau - 2\beta s_{12} - \gamma s_{22} \Bigr)  \nonumber \\
&=\frac{\tau}{\alpha K} \bigg( K - 2\beta \Big( \alpha cd + \gamma ab - 2\beta ad \Big) - \gamma \Big( \alpha d \tau + \gamma b^2 - \alpha bc - 2\beta bd \Big) \bigg). \label{explicicalsone}
\end{align}
Observe from \eqref{definitionk} and \eqref{definitiop} that 
$$
K = (\alpha\gamma - \beta^2)\tau^2 + \big(\gamma b + \beta(a-d) - \alpha c\big)^2.
$$
We decompose $K$ into terms independent of $\alpha$ (denoted by $K_{\text{no-$\alpha$}}$) and terms dependent on $\alpha$ (denoted by $K_{\text{$\alpha$}}$). We calculate
\begin{align*}
K_{\text{no-$\alpha$}}&=-\beta^2 (a+d)^2 + \big(\gamma b+\beta a- \beta d \big)^2 \\
&= (-\beta^2 a^2 - 2 \beta^2 ad - \beta^2 d^2) + ( \gamma^2 b^2 + \beta^2 a^2 + \beta^2 d^2 +2\beta \gamma ab - 2 \beta \gamma bd - 2 \beta^2 ad) \\
& = -4 \beta^2 ad + \gamma^2 b^2  + 2 \beta \gamma ab - 2 \beta \gamma bd,
\end{align*}
and $K_{\alpha} =  \alpha \gamma (a+d)^2 - 2 \alpha c \big(\gamma b + \beta(a-d) \big) + \alpha^2 c^2$.
Consider the subtraction terms in the numerator of \eqref{explicicalsone}:
\begin{align*}
&2\beta \Big( \alpha cd + \gamma ab - 2\beta ad \Big) + \gamma \Big( \alpha d \tau + \gamma b^2 - \alpha bc - 2\beta bd \Big) \\
&=\underbrace{2 \beta \big(  \gamma ab - 2 \beta ad \big) + \gamma (\gamma b^2 - 2 \beta bd)}_{T_{\text{no-}\alpha}} \;+\; \underbrace{ 2 \alpha \beta cd  +\gamma \big( \alpha d (a+d) - \alpha bc\big)}_{T_{\alpha}}.
\end{align*}
The non-$\alpha$ terms cancel:
\begin{align*}
K_{\text{no-$\alpha$}}-T_{\text{no-}\alpha} =  \Big(-4 \beta^2 ad + \gamma^2 b^2  + 2 \beta \gamma a b - 2 \beta \gamma bd \Big)- \Big(	 
2 \beta \gamma ab -4 \beta^2 a d +\gamma^2 b^2 - 2 \beta \gamma bd  \Big)=0.
\end{align*}
For the terms containing $\alpha$:
\begin{align*}
K_{\alpha}-T_{\alpha} &=  \alpha \gamma (a+d)^2 - 2 \alpha c \big( \gamma b + \beta a- \beta d \big) + \alpha^2 c^2 - 2 \alpha \beta cd  -\alpha \gamma d (a+d) + \alpha \gamma bc \\
&= \alpha \bigg( \gamma a (a+d)  +	\alpha c^2 - \gamma bc	- 2 \beta ac \bigg).
\end{align*}
Consequently, it follows from \eqref{explicicalsone} that
$$
s_{11} = \frac{\tau}{\alpha K} \big( K_{\alpha}-T_{\alpha} \big) = \frac{\tau}{K} \Big( \gamma a (a+d)  +	\alpha c^2 - \gamma bc	- 2 \beta ac  \Big),
$$
as desired.
\end{proof}

\begin{example}\label{exam:nonunique}
Examples where a SOHHD for a linear vector field is nonunique are given where the conditions Theorem  \ref{expliformuouoper} overlap. For instance, (i) and (ii) of Theorem \ref{expliformuouoper} overlap, when
\[
\gamma b -\alpha c + \beta(a-d) =0 \quad \text{and} \quad a+d=0.
\]
In this case $S=0$ and 
\[
S= \frac{1}{\alpha\gamma - \beta^2}
\begin{pmatrix}
\gamma a - \beta c & \gamma b - \beta d \\
\gamma b - \beta d & \alpha d - \beta b
\end{pmatrix}= \frac{1}{\alpha\gamma - \beta^2}
\begin{pmatrix}
\gamma a - \beta c & \alpha c -\beta a  \\
\alpha c -\beta a & -\alpha a - \beta b
\end{pmatrix}.
\]
are solutions.
\end{example}

\section*{Acknowledgments}
The research of Haesung Lee was supported by the National Research Foundation of Korea (NRF) grant funded by the Korea government (MSIT) (RS-2025-16070171) and by the Regional Innovation System \& Education (RISE)-Regional Growth Innovation LAB program through the Gyeongbuk RISE Center, funded by the Ministry of Education (MOE) and Gyeongsangbuk-do, Republic of Korea (2026-rise-15-105).\\
The research of Gerald Trutnau was supported by the National Research Foundation of Korea (NRF) grant funded by the Korea government (MSIT) (RS-2025-16070171).

\newpage

\noindent
Haesung Lee\\
Department of Mathematics and Big Data Science  \\
Kumoh National Institute of Technology \\
Gumi, Gyeongsangbuk-do 39177, South Korea \\
E-mail: fthslt@kumoh.ac.kr \\ \\ 
Gerald Trutnau\\
Department of Mathematical Sciences \\
and Research Institute of Mathematics\\
Seoul National University\\
1 Gwanak-ro, Gwanak-gu,
Seoul 08826, South Korea  \\
E-mail: trutnau@snu.ac.kr
\end{document}